\documentclass[11pt,reqno]{amsart}

\usepackage[margin=1.08in]{geometry}
\usepackage{amsmath,amssymb,amsthm,mathtools,mathrsfs}
\usepackage{microtype}
\usepackage{booktabs}
\usepackage{enumitem}
\usepackage{array}
\usepackage{longtable}
\usepackage{xcolor}
\usepackage[hidelinks]{hyperref}

\usepackage[nameinlink,capitalise,noabbrev]{cleveref}

\allowdisplaybreaks
\setlist[itemize]{leftmargin=2em}
\setlist[enumerate]{leftmargin=2em}

\newtheorem{theorem}{Theorem}[section]
\newtheorem{proposition}[theorem]{Proposition}
\newtheorem{lemma}[theorem]{Lemma}
\newtheorem{corollary}[theorem]{Corollary}

\theoremstyle{definition}
\newtheorem{definition}[theorem]{Definition}

\theoremstyle{remark}
\newtheorem{remark}[theorem]{Remark}

\newcommand{\R}{\mathbb R}

\newcommand{\E}{\mathbb E}
\newcommand{\Pp}{\mathbb P}

\newcommand{\tr}{\operatorname{tr}}

\newcommand{\supp}{\operatorname{supp}}
\newcommand{\seca}{\operatorname{Sec}}

\newcommand{\norm}[1]{\left\lVert #1\right\rVert}

\newcommand{\eps}{\varepsilon}

\newcommand*{\prob}[1]{\mathbb{P}\left\{ #1 \right\}}
\newcommand*{\probb}[2]{\mathbb{P}_{#1}\left\{ #2 \right\}}

\title{Randomly Permuted Orthogonal Products and Fast Dimension Reduction}
\author{Rafael Chiclana}
\date{}

\begin{document}

\begin{abstract}
	We study the effect of random signed permutations on products of
	orthogonal matrices and their applications to fast dimension reduction. Let $A,B \in \mathbb{R}^{d\times d}$ be orthogonal matrices and let $\Sigma \in \mathbb{R}^{d\times d}$ be a uniformly random signed permutation matrix. We analyze the random orthogonal matrix
	\[
	U=A \Sigma B,
	\]
	and show that, under mild assumptions on the size of the entries of $A$ and $B$,
	\[
	\max_{i,j=1,\ldots,d} |U_{ij}|
	=O\left (
	\sqrt{\frac{\log d}{d}}\right )
	\]
	with high probability. As an application, we show that ORA, an analogue of the Kac walk in which every update is a \(\pi/4\) rotation, reaches
	the same maximal entry scale after \(O(d\log d)\) updates. This resolves a
	question of Jain et al. and improves the running time of their construction. We
	also show that parallel ORA reaches this scale after \(O(\log d)\) rounds. We then study the random embedding
	\[
	\Phi
	=
	\sqrt{\frac{d}{m}}\,
	P_I U D_{\xi'},
	\]
	where \(P_I\) restricts to \(m\) coordinates and \(\xi'\) is an
	independent Rademacher vector. We identify two parameters controlling
	norm preservation and show that, throughout the corresponding
	admissible range, \(\Phi\) achieves optimal
	embedding dimension $m\asymp\varepsilon^{-2}\log(N)$. Finally, we extend the result to structured infinite models, including sparse vectors, low-rank matrices, and finite unions of subspaces.
\end{abstract}

\maketitle

\section{Introduction}

The Johnson--Lindenstrauss lemma \cite{JohnsonLindenstrauss1984}
provides dimension reduction with optimal target dimension
\[
m\asymp\eps^{-2}\log N,
\]
but applying an \(m\times d\) dense sub-Gaussian matrix requires
\(O(md)\) arithmetic operations. Fast Johnson--Lindenstrauss
constructions seek to retain the optimal target dimension while reducing
the cost of applying the embedding. A common approach is to precondition
the input vectors using a structured orthogonal transformation that
spreads their mass more evenly across the coordinates, and then reduce
the dimension by coordinate sampling or a sparse projection
\cite{AilonChazelle2009,AilonLiberty2009,KrahmerWard2011}.

More recently, Jain, Pillai, Sah, Sawhney, and Smith developed fast and
memory-efficient Johnson--Lindenstrauss embeddings based on the Kac walk
\cite{JainEtAl2022}. The Kac walk is a Markov chain on
\(\mathbb S^{d-1}\) in which, at each step, two coordinates are chosen
uniformly and rotated by an independent uniform angle. The walk rapidly
spreads mass among the coordinates: Pillai and Smith proved that its
total-variation mixing time is of order \(d\log d\)
\cite{PillaiSmith2017}, and these results were significantly
strengthened by Jain and Mizgerd, who recently established cutoff
\cite{JainMizgerd2026}. For a matrix $A$, we refer to $\|A\|_\infty=\max_{i,j}|A_{ij}|$ as its \emph{coherence}. The construction of \cite{JainEtAl2022} runs the Kac walk until the resulting orthogonal matrix \(Q_T\) satisfies 
\begin{equation}\label{eq:near-optimal-coherence} \norm{Q_T}_\infty \lesssim \sqrt{\frac{\log d}{d}},
\end{equation}
which occurs after \(T=O(d\log d)\) updates. This bound allows randomly subsampled rows of $Q_T$ to satisfy the restricted isometry property. Throughout the range \(\log N = O(\eps\sqrt d)\), up to logarithmic factors, their construction achieves optimal embedding dimension with \(O(d\log d)\) running time and \(O(1)\) additional working memory.

The same work also considered orthogonal repeated averaging (ORA), an analogue of the Kac walk in which every selected pair of coordinates is rotated by the angle \(\pi/4\). In this case, their
argument obtains near optimal coherence \eqref{eq:near-optimal-coherence} only after
\[
O(d\log d\log\log d)
\]
updates. Thus, the ORA construction incurs an additional
\(\log\log d\) factor in its first stage. The removal of this factor is identified as the most immediate problem left open in \cite[Section~5]{JainEtAl2022}.

We resolve this problem by showing that ORA reaches the required
coherence scale after \(O(d\log d)\) updates.

\begin{theorem}\label{thm:one-block-ora}
	There exists a universal constant \(C>0\) such that, for every
	$d$ sufficiently large, an ORA walk \(Q_T\) of length $T\geq C d\log d$ satisfies, with probability at least $1-d^{-2}$,
	\begin{equation}\label{eq:6.1}
		\norm{Q_T}_\infty
		\leq
		8\sqrt{\frac{\log d}{d}}.
	\end{equation}
\end{theorem}

The first stage of the ORA construction of \cite{JainEtAl2022} can
therefore be shortened from \(O(d\log d\log\log d)\) to
\(O(d\log d)\) updates. The remaining stages of their construction are
unchanged, yielding the following improvement of their ORA
Johnson--Lindenstrauss embedding.

\begin{theorem}\label{thm:jps-improvement}
	There exists a universal constant \(C>0\) such that the following
	holds. Let \(N\geq d\ge C\), \(0<\eps<C^{-1}\), and let
	\(X\subseteq \R^d\) with \(|X|=N\). Then there
	exists an ORA-based random linear map
	\(\Phi:\R^d\to\R^m\), with \(m\le C\eps^{-2}\log N\), such that
	\[
	(1-\eps)\norm{x}_2
	\le
	\norm{\Phi x}_2
	\le
	(1+\eps)\norm{x}_2
	\qquad
	\text{for every }x\in X,
	\]
	with probability at least \(2/3\). Moreover, \(\Phi x\) can be computed in time
	\[
	O\left(
	d\log d+
	\min\left\{
	d\log d\log N,\,
	\eps^{-2}(\log N)^2(\log\log N)^2(\log d)^4
	\right\}
	\right)
	\]
	using \(O(1)\) additional working memory.
\end{theorem}

The map $\Phi$ is obtained from \cite[Algorithm~2]{JainEtAl2022} by replacing its initial ORA walk of length \(O(d\log d\log\log d)\) by one of length \(O(d\log d)\).

A similar phenomenon occurs for the parallel Kac walk introduced by
Lu, Qin, Song, Yao, and Zhao \cite{LuQinSongYaoZhao2024}. In each
round, the coordinates are partitioned into a uniform perfect matching
and independent random rotations are applied simultaneously to the
matched pairs. The authors asked whether these independent rotations
can be replaced by a common random rotation, or even by a fixed angle
\cite{LuQinSongYaoZhao2024,LuQinSongYaoZhao2025}. We consider here
a version, denoted by \emph{parallel ORA}, for which every matched
pair is rotated by the angle \(\pi/4\). We show that a parallel ORA walk $Q_T$ satisfies \eqref{eq:near-optimal-coherence} after $T=O(\log d)$ parallel rounds.

\begin{theorem}
	\label{thm:parallel-ora-coherence}
	There exists a universal constant \(C>0\) such that, for every sufficiently large even $d$, a parallel ORA walk \(Q^{\mathrm{par}}_T\) with $T\geq C \log d$ rounds satisfies, with probability at least $1-d^{-2}$,
	\[
	\norm{Q_T^{\mathrm{par}}}_\infty
	\leq
	C\sqrt{\frac{\log d}{d}}.
	\]
\end{theorem}

Thus, for the purpose of obtaining near optimal coherence, the
independent random rotations in a parallel Kac round can be replaced by
the fixed angle \(\pi/4\). We do not address the stronger mixing or
pseudorandomness properties considered in \cite{LuQinSongYaoZhao2024,LuQinSongYaoZhao2025}.

The preceding results for ordinary and parallel ORA are consequences of
a more general phenomenon. Let \(A,B\in O(d)\) be orthogonal matrices, let \(D_\xi\) be a Rademacher diagonal matrix, and let \(P_\pi\) be an independent uniform permutation
matrix. We study the random orthogonal matrix
\[
U=AD_\xi P_\pi B.
\]
Our analysis of the bilinear forms \(\langle y,Ux\rangle\) shows that
the random signed permutation prevents large entries of the two
orthogonal factors from aligning. As a consequence,
we obtain the following coherence amplification principle: Under a mild
assumption on the maximal entries of \(A\) and \(B\), the matrix \(U\) has
near optimal coherence with high probability.

\begin{corollary}\label{cor:coherence-amplification}
	Let \(A,B\in O(d)\), let \(D_\xi \in \mathbb{R}^{d\times d}\) be a Rademacher diagonal, and
	let \(\pi\) be an independent uniform permutation on \(S_d\). Let \(0<\eta<1/2\) and set $U=AD_\xi P_\pi B$. Suppose that
	\[
	\norm{A}_\infty\norm{B}_\infty
	\le
	\frac{1}{\sqrt{d \log(d/\eta)}}.
	\]
	Then, with probability at least \(1-\eta\),
	\[
	\norm{U}_\infty
	\le
	4\frac{\sqrt{\log(d/\eta)}}{\sqrt{d}}.
	\]
\end{corollary}

The same permutation mechanism can also be used to control norm
preservation after coordinate sampling. Let \(I\subseteq [d]\) with
\(|I|=m\), and consider the random map
\[
\Phi
=
\sqrt{\frac dm}\,
P_IAD_\xi P_\pi BD_{\xi'},
\]
where \(\xi,\xi'\) are independent Rademacher vectors and \(\pi\) is
an independent uniform permutation.

The behavior of the two orthogonal factors is quantified by the
\emph{sampled column energy}
\[
\kappa_I(A)
=
\frac dm
\max_{k\in[d]}
\sum_{j\in I}A_{jk}^2
\]
and the \emph{maximal variance}
\[
\Theta_B(X)
=
d\sup_{\substack{x\in X\\k\in[d]}}
\sum_{j=1}^d B_{kj}^2x_j^2.
\]
The first parameter measures the largest energy of a column of \(A\) carried by the selected rows, while the second controls the largest
variance of a coordinate of \(BD_{\xi'}x\), uniformly over
\(x\in X\).

\begin{theorem}\label{thm:finite-JL}
	Let \(X\subseteq \mathbb S^{d-1}\) with \(|X|=N\), let
	\(0<\eta<1/2\), and assume that \(1\le m\le d\). Set
	\begin{equation}\label{eq:finite-logs}
		\ell
		=
		\log\left(\frac{8N}{\eta}\right),
		\qquad
		L_d
		=
		\log\left(\frac{8Nd}{\eta}\right),
		\qquad
		L_m
		=
		\log\left(\frac{8Nm}{\eta}\right).
	\end{equation}
	Suppose the admissibility condition
	\begin{equation}\label{eq:finite-capacity}
		36\kappa_I(A)\Theta_B(X)\,mL_dL_m
		\le
		d.
	\end{equation}
	Then, with probability at least \(1-\eta\),
	\begin{equation}\label{eq:finite-conclusion}
		\sup_{x\in X}
		\left|
		\norm{\Phi x}_2-1
		\right|
		\le
		10\sqrt{\frac{\ell}{m}}.
	\end{equation}
\end{theorem}

In particular, taking $m\asymp\eps^{-2}\log(N/\eta)$ gives optimal embedding dimension throughout the admissible range
\[
\kappa_I(A)\Theta_B(X)\,
m
\log\left(\frac{Nd}{\eta}\right)
\log\left(\frac{Nm}{\eta}\right)
\lesssim d.
\]
When $\eta$ is fixed and \(\log N\) dominates \(\log d\),
one has \(L_d,L_m\asymp\log N\), and the admissibility condition becomes
\[
\kappa_I(A)\Theta_B(X)\,
\eps^{-2}(\log N)^3
\lesssim d.
\]
Since \(\kappa_I(A)\Theta_B(X)\ge1\), the largest cardinality range
covered by this criterion is
\[
N
\le
\exp\left(
C\eps^{2/3}d^{1/3}
\right).
\]
This scale is reached when both local parameters are bounded by universal constants.

\begin{remark}\label{rem:remark}
	The appearance of the cube-root scale is not merely an artifact of
	the proof, although it should not be interpreted as a limitation for
	every orthogonal ensemble. In \Cref{sec:limitations}, we construct orthogonal matrices \(A,B \in O(d)\) and a vector \(x\in\mathbb S^{d-1}\) satisfying
	\[
	\kappa_I(A)=\Theta_B(\{x\})=1,
	\qquad
	m=\sqrt d,
	\]
	for which constant distortion occurs with probability at least
	\[
	\exp\left[
	-Cd^{1/3}(\log d)^{2/3}
	\right].
	\]
	Thus, up to logarithmic factors, a general theorem expressed only in
	terms of \(\kappa_I(A)\) and \(\Theta_B(X)\) cannot, in general,
	surpass the cube-root range. Stronger concentration for particular ensembles must exploit additional structure not captured by these two parameters.
\end{remark}

Finally, in Section \ref{sec:infinite-sets} we extend Theorem \ref{thm:finite-JL} to several structured infinite
models. These include sparse vectors, low-rank matrices, and finite unions
of subspaces, with the optimal number of rows throughout the
corresponding admissible range. More generally, for a broad class of
structured cones, we obtain optimal embedding dimension governed by
their Gaussian width.

The paper is organized as follows. In \Cref{sec:tools}, we introduce
the notation and collect the scalar and matrix concentration estimates
for randomly permuted sums used throughout the paper.
\Cref{sec:bilinear-overlap} studies bilinear forms associated with
randomly permuted orthogonal products and proves the coherence
amplification principle \Cref{cor:coherence-amplification}. In \Cref{sec:finite},
we prove \Cref{thm:finite-JL} and develop the sampled column energy and
maximal variance framework for finite sets. In
\Cref{sec:infinite-sets}, we introduce quadratic norming families and
extend our results to structured infinite models.
\Cref{sec:mixers} applies this framework to specific orthogonal
constructions, with particular emphasis on the Kac walk, ORA, and
parallel ORA, and contains the proofs of the corresponding results
stated above. Finally, \Cref{sec:limitations} gives an obstruction
showing that the sampled column energy and maximal variance alone
cannot surpass the cube-root barrier discussed in Remark \ref{rem:remark}.
Auxiliary permutation, coupling, and moment arguments are deferred to
the appendices.

\section{Notation and preliminary results}\label{sec:tools}

\subsection{Notation}

For \(d\in\mathbb N\), we write \([d]=\{1,\ldots,d\}\), and denote by
\(S_d\) the permutation group on \([d]\). We write
\(\mathbf 1=(1,\ldots,1)\in\R^d\). For \(x\in\R^d\) and
\(\pi\in S_d\), set 
\[x_\pi=(x_{\pi(1)},\ldots,x_{\pi(d)}),\] 
and let \(P_\pi\) be the permutation matrix satisfying \(P_\pi x=x_\pi\).
Thus, \((P_\pi M)_{kj}=M_{\pi(k),j}\).

For a vector \(x=(x_1,\ldots,x_d)\in\R^d\), and
\(1\le p<\infty\), we write
\[
\norm{x}_p
=
\left(\sum_{j=1}^d |x_j|^p\right)^{1/p},
\qquad
\norm{x}_\infty
=
\max_{j\in[d]}|x_j|,
\qquad \|x\|_0=|\{j:x_j\neq0\}|.
\]
The unit Euclidean sphere of \(\R^d\) is denoted by
\(\mathbb S^{d-1}\). For \(x,y\in\R^d\), we write $\langle x,y\rangle=\sum_{j=1}^d x_jy_j$ for the Euclidean inner product. 

We denote by \(I_d\) the \(d\times d\) identity matrix and, for
\(x\in\R^d\), write \(D_x=\operatorname{diag}(x_1,\ldots,x_d)\). For a real matrix
\(M\), we write \(M_{ij}\) for the entry of \(M\) in row \(i\) and
column \(j\), and \(M^{\mathsf T}\) for its transpose. The
orthogonal group is denoted by
\(O(d)=\{A\in\R^{d\times d}:A^{\mathsf T}A=I_d\}\). If \(M\) and
\(N\) are symmetric matrices of the same size, we write \(M\succeq0\)
when \(M\) is positive semidefinite, and \(M\preceq N\) when
\(N-M\succeq0\).

For a real matrix \(M\in\R^{m\times d}\), its operator, Frobenius, and maximum
norms are
\[
\norm{M}
=
\sup_{x\in\mathbb S^{d-1}}\norm{Mx}_2,
\qquad
\norm{M}_F
=
\left(\sum_{i,j}|M_{ij}|^2\right)^{1/2},
\qquad
\|M\|_\infty = \max_{i,j} |M_{ij}|,
\]
respectively. We often refer to $\|M\|_\infty$ as the coherence of $M$. The trace of a square matrix \(M\) is denoted by
\(\tr M=\sum_iM_{ii}\), and its permanent by
\begin{equation}\label{eq:per}
\operatorname{per}(M)
=
\sum_{\pi\in S_d}
\prod_{i=1}^d M_{i,\pi(i)}.
\end{equation}
For matrices \(M,N\) of the same size,
we use the same notation for the Frobenius inner product
\[
\langle M,N\rangle=\tr(M^{\mathsf T}N).
\]
If \(\mu\) is a probability measure, then \(\Pp_\mu\) and \(\E_\mu\)
denote probability and expectation with respect to \(\mu\). For
\(1\le p<\infty\),
\[
\norm{X}_{L^p(\mu)}
=
\left(\E_\mu |X|^p\right)^{1/p}.
\]
When the underlying measure is clear, we simply write
\(\norm{X}_{L^p}\). We use subscripts such as \(\Pp_\pi\),
\(\E_\pi\), and \(L^p(\pi)\) when the randomness comes from a uniform
permutation. We write \(X\stackrel{\mathrm d}=Y\) to denote equality in distribution.

\subsection{Concentration inequalities for permuted sums}

A central object in this work is a randomly permuted sum of the form
\[
\sum_{k=1}^d a_k b_{\pi(k)},
\]
where \(a,b\in\R^d\) are deterministic vectors and
\(\pi\) is uniformly distributed on \(S_d\). Concentration inequalities for arrays
of the form
\[
\sum_{k=1}^d c_{k,\pi(k)}
\]
were obtained by Albert \cite{Albert2019}; see also
\cite{Barber2024} for related concentration inequalities for weighted
sums of exchangeable random variables. We will only require the simple case
\(c_{ij}=a_i b_j\). For completeness, we provide a particularly short proof in Appendix \ref{app:scalar-permutation-mgf}.

\begin{lemma}\label{lem:scalar-comb-bern}
	Let \(a,b\in\R^d\) satisfy $\sum_i b_i=0$ and let \(\pi\) be uniform on \(S_d\). Set
	\[
	\sigma^2
	=
	\frac{\norm{a}_2^2\norm{b}_2^2}{d},
	\qquad
	M
	=
	\norm{a}_\infty\norm{b}_\infty.
	\]
	Then
	\begin{equation}\label{eq:scalar-permutation-mgf}
		\log \E_\pi
		\exp\left(
		\lambda \sum_{k=1}^d a_k b_{\pi(k)}
		\right)
		\le
	4\lambda^2 \sigma^2 \qquad \text{for } |\lambda|\leq \frac{1}{M}
	\end{equation}
\end{lemma}

We next record the specialization of a more general matrix-valued
permutation inequality due to Mackey, Jordan, Chen, Farrell, and Tropp
\cite[Corollary~10.3]{MackeyEtAl2014} that will be needed below.

\begin{lemma}\label{lem:mackey-comb}
	Let \(b\in\R^d\) satisfy \(\sum_{k=1}^d b_k=0\), let
	\(X_1,\ldots,X_d\) be symmetric \(r\times r\) matrices, and let
	\(\pi\) be uniform on \(S_d\). Set
	\[
	\sigma^2
	=
	\frac{\norm{b}_2^2}{d}
	\left\|
	\sum_{j=1}^d X_j^2
	\right\|,
	\qquad
	M
	=
	\norm{b}_\infty\max_{j\in[d]}\norm{X_j}.
	\]
	Then
	\[
	\probb{\pi}{\norm{\sum_{j=1}^d b_{\pi(j)}X_j}\ge t}
	\le
	2r\exp\left(
	-\frac{t^2}
	{12\sigma^2+4\sqrt2\,Mt}
	\right)
	\qquad \text{for }t\geq0.
	\]
\end{lemma}

Indeed, this follows from their matrix-array result by taking
\(A_{jk}=b_kX_j\). More recent concentration inequalities for
exchangeable matrix-valued and tensor-valued arrays appear in \cite{ChengBarber2026}. The preceding result is sufficient for all the arguments in this paper.

We will also use the following convex Lipschitz concentration inequality due to Talagrand; see \cite[p.181]{Talagrand2021}.

\begin{lemma}[Talagrand's inequality]
	\label{lem:talagrand}
	Let \(\xi\in\{-1,1\}^d\) be a Rademacher vector, and let
	\(f:\mathbb R^d\to\mathbb R\) be convex and \(L\)-Lipschitz with
	respect to the Euclidean norm. If \(M_f\) is a median of \(f(\xi)\),
	then
	\[
	\Pp\left\{
	\left|f(\xi)-M_f\right|\ge t
	\right\}
	\le
	4\exp\left(
	-\frac{t^2}{8L^2}
	\right) \qquad \text{for } t\geq0.
	\]
\end{lemma}

Finally, we record the standard Bernstein tail and moment estimates that
follow from a local sub-Gaussian bound on the moment-generating function;
see, for example, \cite[Sections~2.8--2.9]{Vershynin2026}.

\begin{lemma}\label{lem:local-mgf}
	Let \(X\) be a random variable and let \(\sigma,M>0\). Suppose that
	\[
	\log \E e^{\lambda X}
	\le
	\lambda^2\sigma^2
	\qquad
	\text{for }|\lambda|\le\frac1M.
	\]
	Then, for every \(t\ge0\),
	\[
	\Pp\{|X|\ge t\}
	\le
	2\exp\left[
	-\min\left\{
	\frac{t^2}{4\sigma^2},
	\frac{t}{2M}
	\right\}
	\right].
	\]
\end{lemma}

\begin{proof}
	For \(0\le\lambda\le1/M\), Markov's inequality and the hypothesis give
	\[
	\Pp\{X\ge t\} = \prob{e^{\lambda X} \geq e^{\lambda t}} \leq \frac{\mathbb{E} e^{\lambda X}}{e^{\lambda t}}
	\le
	\exp\left(\lambda^2\sigma^2 - \lambda t\right).
	\]
	If \(t\le2\sigma^2/M\), then taking $\lambda = t/(2\sigma^2)$ yields \(\Pp\{X\ge t\}\le \exp \left (-t^2/(4\sigma^2) \right )\). Otherwise, taking $\lambda = 1/M$ gives
	\(\Pp\{X\ge t\}\le \exp \left (-t/(2M)\right )\). Applying the same argument to \(-X\) and taking a union bound completes the proof.
\end{proof}

\section{Permutation concentration and amplification}
\label{sec:bilinear-overlap}

Let \(A,B\in O(d)\) be orthogonal matrices, let $D_\xi \in \mathbb{R}^{d\times d}$ be a Rademacher diagonal, and let \(P_\pi \in \mathbb{R}^{d\times d}\) be the permutation matrix associated with an independent uniform permutation \(\pi\in S_d\), defined by
\[
(P_\pi x)_k=x_{\pi(k)}.
\] 
Our object of interest is the random orthogonal matrix
\begin{equation}\label{eq:U}
U=AD_\xi P_\pi B \in O(d).
\end{equation}

Our first result combines the standard sub-Gaussian concentration of
Rademacher sums with concentration inequalities for randomly permuted
sums to obtain concentration estimates for bilinear forms $\langle y,Ux\rangle$, where \(x,y\in\mathbb S^{d-1}\). The idea is that, conditionally
on the permutation $\pi$, the random variable $\langle y,Ux\rangle$ is a Rademacher sum with variance
\[
\sum_{k=1}^d
(A^{\mathsf T}y)_k^2
(Bx)_{\pi(k)}^2.
\]
Note that this conditional variance is itself a random permuted sum, and \Cref{lem:scalar-comb-bern} can be used to analyze its fluctuation around its mean $1/d$. 

\begin{theorem}\label{thm:bilinear-overlap}
	Let \(A,B\in O(d)\), let $D_\xi \in \mathbb{R}^{d\times d}$  be a Rademacher
	diagonal, and let \(\pi\) be an independent uniform permutation on
	\(S_d\). Set
	\[
	U=AD_\xi P_\pi B.
	\]
	Then, for every $x, y \in \mathbb{S}^{d-1}$,
	\[
	\log \E_{\xi,\pi} \left ( e^{\lambda\langle y,Ux\rangle} \right )
	\le
	\frac{3\lambda^2}{2d} \qquad \text{for } |\lambda| \leq \frac{1}{\norm{A^{\mathsf T}y}_\infty
		\norm{Bx}_\infty}.
	\]
\end{theorem}

\begin{proof}
	Set
	\[
	a=A^{\mathsf T}y,
	\qquad
	b=Bx.
	\]
	Since \(A\) and \(B\) are orthogonal, we have that $\norm{a}_2=\norm{b}_2=1$. Observe that
	\[
	\langle y,Ux\rangle
	=
	\sum_{k=1}^d a_k b_{\pi(k)} \xi_k.
	\]
	Conditionally on \(\pi\), independence of the Rademacher variables and the inequality $\cosh t \leq e^{t^2/2}$ give
	\begin{equation}\label{eq:3.1.1}
	\E_\xi \left (
	e^{
	\lambda\langle y,Ux\rangle}\right ) = \prod_{k=1}^d \mathbb{E}_{\xi_k} \left (e^{\lambda a_k b_{\pi(k)} \xi_k} \right ) = \prod_{k=1}^d \cosh(\lambda a_k b_{\pi(k)})
	\le
	\exp \left (\frac{\lambda^2}{2}
	\sum_{k=1}^d a_k^2b_{\pi(k)}^2 \right ).
	\end{equation}
	Define vectors $u, v \in \mathbb{R}^d$ coordinatewise by
	\[
	u_k=a_k^2,
	\qquad
	v_k=b_k^2 - \frac{1}{d}.
	\]
	We notice that 
	\begin{equation}\label{eq:3.1.2}
	\sum_{k=1}^d a_k^2b_{\pi(k)}^2
	=
	\frac1d+
	\sum_{k=1}^d u_kv_{\pi(k)},
	\end{equation}
	 and thus it suffices to analyze the random permuted sum $\sum_k u_k v_{\pi(k)}$, which can be handled by Lemma \ref{lem:scalar-comb-bern}. First, observe that $\sum_k v_k = 0$. Moreover, 
	\[
	\norm{u}_2^2
	=
	\sum_{k=1}^d a_k^4
	\le
	\norm{a}_\infty^2\|a\|_2^2 = \|a\|_\infty^2.
	\]
	A similar argument yields $\|v\|_2^2\leq \|b\|^2_\infty$. Also, since $\|b\|_2^2 = 1$, we have $\|b\|_\infty^2\geq 1/d$, and therefore
	\[
	\norm{u}_\infty
	=
	\norm{a}_\infty^2,
	\qquad
	\norm{v}_\infty
	\leq
	\norm{b}_\infty^2.
	\]
	Applying \Cref{lem:scalar-comb-bern}, together with the estimates above,
	gives
	\[
	\E_\pi
	\exp\left(
	s\sum_{k=1}^d u_kv_{\pi(k)}
	\right)
	\le
	\exp\left(
	\frac{4s^2\norm{a}_\infty^2\norm{b}_\infty^2}{d}
	\right) \qquad \text{for } |s|
	\le
	\frac{1}{\norm{a}_\infty^2\norm{b}_\infty^2}.
	\]
	Therefore, using \eqref{eq:3.1.1} and \eqref{eq:3.1.2}, and taking $s=\lambda^2/2$,
	\[
	\begin{aligned}
		\E_{\pi,\xi}
		e^{\lambda\langle y,Ux\rangle}
		&\le
		\E_\pi
		\exp\left(
		\frac{\lambda^2}{2}
		\sum_{k=1}^d a_k^2b_{\pi(k)}^2
		\right)=
		e^{\lambda^2/(2d)}
		\E_\pi
		\exp\left(
		\frac{\lambda^2}{2}
		\sum_{k=1}^d u_kv_{\pi(k)}
		\right)\\
		&\le
		\exp\left(
		\frac{\lambda^2}{2d}
		+
		\frac{\lambda^4\norm{a}_\infty^2\norm{b}_\infty^2}{d}
		\right)
		\le
		\exp\left(
		\frac{3\lambda^2}{2d}
		\right),
	\end{aligned}
	\]
	for $|\lambda|\leq 1/(\|a\|_\infty \|b\|_\infty)$. Taking logarithms completes the proof.
\end{proof}

Applying \Cref{thm:bilinear-overlap} to the canonical vectors yields
\Cref{cor:coherence-amplification}. We now give its proof.

\begin{proof}[Proof of Corollary \ref{cor:coherence-amplification}]
	Set \(L=\log(d/\eta)\). For fixed \(i,j\in[d]\), we can write
	\[
	U_{ij}
	=
	\langle e_i,Ue_j\rangle.
	\]
	By \Cref{thm:bilinear-overlap,lem:local-mgf},
	\[
	\Pp_{\xi,\pi}\{|U_{ij}|>t\}
	\le
	2\exp\left(
	-\min\left\{
	\frac{dt^2}{6},
	\frac{t}{2\|A\|_\infty \|B\|_\infty}
	\right\}
	\right)
	\qquad\text{for }t\ge0,
	\]
	where we have used that
	\[
	\norm{A^{\mathsf T}e_i}_\infty\le\norm A_\infty,
	\qquad
	\norm{Be_j}_\infty\le\norm B_\infty.
	\]
	Taking $t= 4\sqrt{L/d}$ and using the hypothesis gives
	\[
	\Pp_{\xi,\pi}
	\left\{
	|U_{ij}|>
	4\sqrt{\frac Ld}
	\right\}
	\le
	2e^{-2L}.
	\]
	Finally, a union bound over the \(d^2\) entries of $U$ gives
	\[
	\Pp_{\xi,\pi}
	\left\{
	\norm U_\infty>
	4\sqrt{\frac Ld}
	\right\}
	\le
	2d^2e^{-2L}
	=
	2\eta^2
	\le
	\eta.\qedhere
	\]
\end{proof}

In \Cref{sec:mixers}, we use this result to amplify weak entrywise bounds for ordinary and parallel ORA into nearly optimal coherence.

In the proof of Theorem \ref{thm:finite-JL}, we will need to
control the same random permuted sum uniformly over an \(m\)-dimensional
subspace. This leads to the following matrix-valued analogue. Recall that \( P\in\R^{d\times d}\) is an orthogonal projection
if, and only if,
\[
 P^2= P
\qquad\text{and}\qquad
 P^{\mathsf T}= P.
\]

\begin{lemma}\label{lem:weighted-projection}
	Let \(P\) be a rank-\(m\) orthogonal projection on \(\R^d\), let $\rho=\max_{j\in[d]} P_{jj}$, and let \(q\) be a probability measure on \([d]\). Set
	\[ \sigma^2 = \frac{\rho \|q\|_2^2}{d}, \qquad M=\rho\|q\|_\infty.\]
	If \(\pi\) is uniform on \(S_d\), then
	\[
	\probb{\pi}{\norm{PD_{q_\pi}P - \frac{1}{d}P}\geq t}
	\le
	2m\exp\left(
	-\frac{t^2}
	{12\sigma^2+4\sqrt2\,Mt}
	\right)
	\qquad \text{for }t\geq0.
	\]
\end{lemma}

\begin{proof}
	Consider \(b\in\R^d\) and matrices $X_1,\ldots,X_d$  given by
	\[
	b=q-\frac1d\mathbf 1,
	\qquad
	X_j=Pe_je_j^{\mathsf T}P.
	\]
	Since \(P^2=P\) and $\sum_j e_j e_j^{\mathsf T}=I_d$, we have
	\[
	\begin{aligned}
		\sum_{j=1}^d b_{\pi(j)}X_j
		&=
		P\left(
		\sum_{j=1}^d
		\left(q_{\pi(j)}-\frac1d\right)e_je_j^{\mathsf T}
		\right)P=
		PD_{q_\pi}P-\frac1dPI_dP = PD_{q_\pi}P-\frac1dP.
	\end{aligned}
	\]
	Moreover, $\sum_j b_j=0$ and $X_j$ are symmetric. In view of \Cref{lem:mackey-comb}, it remains to estimate the
	corresponding parameters. Since \(q\) is a probability measure,
	\[
	\norm{b}_2^2
	=
	\norm{q}_2^2-\frac1d
	\le
	\norm{q}_2^2,
	\qquad
	\norm{b}_\infty
	\le
	\norm{q}_\infty.
	\]
	Moreover,
	\[
	\norm{X_j}
	=\norm{Pe_j}_2^2
	=e_j^{\mathsf T} P^{\mathsf T} P e_j 
	=e_j^{\mathsf T}Pe_j
	=P_{jj}
	\le\rho.
	\]
	Finally, noting that $X_j^2=P_{jj}X_j$ and that $X_j$ is positive semidefinite, it follows that
	\[
	\sum_{j=1}^dX_j^2
	\preceq
	\rho\sum_{j=1}^dX_j
	=
	\rho P.
	\]
	The parameters in \Cref{lem:mackey-comb} are thus bounded by
	\[
	\frac{\norm{b}_2^2}{d}
	\left\|\sum_{j=1}^dX_j^2\right\|
	\le
	\frac{\norm{q}_2^2}{d}\|\rho P \| = \frac{\norm{q}_2^2}{d}\rho
	=
	\sigma^2
	\]
	and
	\[
	\norm{b}_\infty\max_{j\in[d]}\norm{X_j}
	\le
	\norm{q}_\infty \rho
	=
	M.
	\]
	Since the matrices \(X_j\) are supported on the \(m\)-dimensional range of \(P\), the result follows from \Cref{lem:mackey-comb}.
\end{proof}

\section{Finite-set embeddings from permuted orthogonal products}\label{sec:finite}

Let \(I\subseteq [d]\) with \(|I|=m\), and let
\(P_I\in\mathbb R^{m\times d}\) denote restriction to the coordinates
in \(I\). Let \(A,B\in O(d)\), let
\(\xi,\xi'\in\{-1,1\}^d\) be independent Rademacher vectors, and let
\(\pi\) be an independent uniform permutation on \(S_d\).

Our object of interest is the random embedding $\Phi\colon \mathbb{R}^d \longrightarrow \mathbb{R}^m$ defined by
\begin{equation}\label{eq:Phi-main}
	\Phi
	=
	\sqrt{\frac dm}\,
	P_I A D_\xi P_\pi B D_{\xi'}.
\end{equation}

Recall that $U=AD_\xi P_\pi B$ is the random orthogonal matrix analyzed in Section \ref{sec:bilinear-overlap}. The map $\Phi$ composes $U$ with an independent Rademacher diagonal $D_{\xi'}$, and then restricts to the coordinates in $I$. We will see that the ability of $\Phi$ to preserve Euclidean distances can be captured by the following two local parameters of $A$ and $B$.

\begin{definition}\label{def:local-parameters}
	Let \(I\subseteq [d]\) with \(|I|=m\), let $A,B \in O(d)$, and let \(X\subseteq \mathbb S^{d-1}\). We define the \emph{sampled column energy}
	of \(A\) by
	\begin{equation}\label{eq:kappa}
		\kappa_I(A)
		:=
		\frac dm
		\max_{k\in[d]}
		\sum_{j\in I}A_{jk}^2
	\end{equation}
	and the \emph{maximal variance} of \(B\) on \(X\) by
	\begin{equation}\label{eq:Theta}
		\Theta_B(X)
		:=
		d\sup_{\substack{x\in X\\k\in[d]}}
		\sum_{j=1}^dB_{kj}^2x_j^2.
	\end{equation}
\end{definition}

The parameter \(\kappa_I(A)\) is the largest energy of a column of
\(A\) carried by the selected rows, normalized by its average \(m/d\).
The parameter \(\Theta_B(X)\) controls the largest Rademacher variance
created by the right factor, since
\[
\E_{\xi'}(BD_{\xi'}x)_k^2
=
\sum_{j=1}^dB_{kj}^2x_j^2.
\]
Averaging the quantities in \eqref{eq:kappa} and \eqref{eq:Theta} over
\(k \in [d]\) shows that
\[
\kappa_I(A)\ge1,
\qquad
\Theta_B(X)\ge1.
\]
Moreover,
\[
\kappa_I(A)
\le
d\|A\|_\infty^2,
\qquad
\Theta_B(X)
\le
d\|B\|_\infty^2.
\]
Thus, these parameters may be viewed as local versions of the global
coherence: \(\kappa_I(A)\) depends only on the energy carried by the
selected rows of \(A\), while \(\Theta_B(X)\) measures the largest
variance generated by \(B\) on the prescribed set \(X\).

\begin{proof}[Proof of Theorem \ref{thm:finite-JL}]
	For convenience we write $\kappa=\kappa_I(A)$ and $\Theta=\Theta_B(X)$. Set
	\[
	Q=A^{\mathsf T}P_I^{\mathsf T}P_IA.
	\]
	One can easily see that $Q^{\mathsf{T}}=Q$ and $Q^2=Q$. Thus, \(Q\) is an orthogonal projection of rank \(m\). Also, by the definition of $\kappa$, 
	\begin{equation}\label{eq:Q-diagonal}
		\rho
		:=
		\max_{j\in[d]}Q_{jj}
		=
		\frac md\kappa.
	\end{equation}
	
	\medskip
	\noindent\emph{Step 1: Inner layer flattening.} Fix $x \in X$ and denote
	\[
	z
	=
	BD_{\xi'}x,
	\qquad
	q
	=
	(z_1^2,\ldots,z_d^2).
	\]
	We suppress the dependence of $x$ to improve clarity in the notation. The first step is to show that, with probability at least $1-\eta/4$ over $\xi'$,
	\begin{equation}\label{eq:main1}
		\max_{x \in X} \|q\|_\infty \leq \frac{2\Theta L_d}{d}.
	\end{equation}
	For every \(x\in X\) the coordinates of $z$ are given by
	\[
	z_k= (BD_{\xi'}x)_k
	=
	\sum_{j=1}^d B_{kj}\xi'_jx_j.
	\]
	Thus, they are centered Rademacher sums whose variances are at most
	\(\Theta/d\). By Hoeffding's inequality \cite[Theorem~2.2.5]{Vershynin2026},
	\[
	\Pp_{\xi'}
	\left\{
	|z_k|>t
	\right\}
	\le
	2\exp\left(
	-\frac{dt^2}{2\Theta}
	\right).
	\]
	Taking \(t^2=2\Theta L_d/d\) and applying a union bound over the set $X\times\{1,\ldots,d\}$ yields (\ref{eq:main1}). We denote this event
	\[ E=\left \{\max_X \|q\|_\infty \leq \frac{2\Theta L_d}{d}\right \}. \]
	
	\medskip
	\noindent\emph{Step 2: Operator norm control by random permutation.}
	Condition on the event \(E\). For every \(x\in X\), define
	\[
	y=P_\pi z=z_\pi,
	\qquad
	T_x=\sqrt{\frac dm}\,P_IAD_y.
	\]
	Since \(D_\xi y=D_y\xi\), we have
	\[
	\Phi x=T_x\xi.
	\]
	The second step is to show that, outside an event of probability at
	most \(\eta/4\) over \(\pi\),
	\begin{equation}\label{eq:main2}
		\max_{x\in X}\norm{T_x}^2
		\le
		\frac2m.
	\end{equation}
	We start by noting that
	\begin{align}\label{eq:4.2.1}
		\norm{T_x}^2
		&= \|T_x^{\mathsf T} T_x\| = 
		\frac{d}{m} \|D_y Q D_y\| = \frac{d}{m} \|D_y Q^{\mathsf T} Q D_y\| 
		= \frac{d}{m} \|Q D_{y}D_y Q^{\mathsf T}\| = \frac{d}{m} \|Q D_{q_\pi} Q\|,
	\end{align}
	where we have used that $Q$ is an orthogonal projection, and thus $Q^{\mathsf T}=Q$ and $Q^2=Q$, and the general fact that $\|A^{\mathsf T}A\|=\|AA^{\mathsf T}\|$. Moreover, \(q\) is a probability measure on \([d]\) satisfying
	\[
	\norm{q}_2^2
	\le
	\norm{q}_\infty
	\le
	\frac{2\Theta L_d}{d}
	\]
	for any realization of $\xi'$ belonging to $E$. Hence, the parameters in \Cref{lem:weighted-projection} satisfy
	\[
	\sigma^2
	\le
	\frac{2\rho\Theta L_d}{d^2},
	\qquad
	M
	\le
	\frac{2\rho\Theta L_d}{d}.
	\]
	Applying \Cref{lem:weighted-projection} with \(t=1/d\) gives
	\[
	\Pp_\pi
	\left\{
	\norm{QD_{q_\pi}Q-\frac1dQ}>\frac1d
	\right\}
	\le
	2m\exp\left(
	-\frac{1}{36\rho\Theta L_d}
	\right).
	\]
	Recall that \(\kappa=\frac{d}{m} \rho\). Thus, the admissibility condition (\ref{eq:finite-capacity}) implies
	\[
	36\rho\Theta L_dL_m\leq 1.
	\]
	Consequently, by the triangle inequality,
	\[
	\Pp_\pi
	\left\{
	\norm{QD_{q_\pi}Q}>\frac2d
	\right\}
	\leq
	\Pp_\pi
	\left\{
	\norm{QD_{q_\pi}Q-\frac1dQ}>\frac1d
	\right\}
	\le 2m \exp(-L_m) \leq \frac{\eta}{4N}.
	\]
	In view of (\ref{eq:4.2.1}), a union bound over \(X\) yields (\ref{eq:main2}) with probability at least $1-\eta/4$. We denote this event by
	\[
	F=
	\left\{
	\max_{x\in X}\norm{T_x}^2\le\frac2m
	\right\}.
	\]

	\medskip
	\noindent\emph{Step 3: Control of the conditional mean by random permutation.}
	Condition on the event $E$. The third step is to show that, outside an event of probability at most
	\(\eta/4\) over \(\pi\),
	\begin{equation}\label{eq:main3}
		\max_{x\in X}
		\left|
		\norm{T_x}_F^2-1
		\right|
		\le
		\frac{1}{\sqrt m},
	\end{equation}
	We start by noting that
	\[
	\norm{T_x}_F^2
	=
	\tr(T_x^{\mathsf T}T_x)
	= \frac{d}{m}\tr (D_y Q D_y) = \frac{d}{m} \tr(QD_{y}^2)=
	\frac dm\sum_{k=1}^dQ_{kk}q_{\pi(k)}.
	\]
	This is a random permuted sum that can be studied with Lemma \ref{lem:scalar-comb-bern} after proper centering. Define $a, b\in \mathbb{R}^d$ with coordinates given by
	\[ a_k=\frac{Q_{kk}}{m},\qquad b_k=dq_k-1.\]
	Since $q$ is a probability measure, we have $\sum_k b_k=0$. Furthermore, since $\tr (Q) =m$ we have
	\begin{equation}\label{eq:4.2.3.1}
		\sum_{k=1}^d a_k b_{\pi(k)} = \frac{d}{m} \sum_{k=1}^d Q_{kk} q_{\pi(k)} - \sum_{k=1}^d\frac{Q_{kk}}{m} = \|T_x\|_F^2 - 1.
	\end{equation}
	Next, we estimate the parameters $\sigma^2$ and $M$ appearing in Lemma \ref{lem:scalar-comb-bern}. First, since $Q_{kk}\leq \rho$ and $\rho=\frac{m}{d} \kappa$, we obtain
	\[ \|a\|_2^2 \leq \sum_{k=1}^d \frac{Q_{kk}^2}{m^2} \leq \frac{\rho}{m^2} \tr(Q)=\frac{\rho}{m}=\frac{\kappa}{d}, \qquad \|a\|_\infty\leq \frac{\rho}{m}=\frac{\kappa}{d}. \]
	On the other hand,  since $q$ is a probability measure, conditioning on the event $E$ we have
	\[\|b\|_2^2 \leq d^2\|q\|_2^2 \leq d^2\|q\|_\infty \leq 2d\Theta L_d.\]
	Finally, using that $\Theta\geq1$ and $L_d\geq 1$ we get
	\[ \|b\|_\infty \leq \max \{d\|q\|_\infty,1\} \leq 2\Theta L_d.\]
	In conclusion, the parameters appearing in Lemma \ref{lem:scalar-comb-bern} satisfy
	\[ \sigma^2 \leq \frac{2\kappa \Theta L_d}{d},\qquad M\leq \frac{2\kappa \Theta L_d}{d}.\]
	Moreover, the admissibility condition \eqref{eq:finite-capacity} gives
	\[ \frac{2\kappa \Theta L_d}{d} \leq \frac{1}{18mL_m}.\]
	Hence, by \Cref{lem:scalar-comb-bern,lem:local-mgf}, for every
	\(t\ge0\),
	\[
	\Pp_\pi
	\left\{
	\left|
	\norm{T_x}_F^2-1
	\right|>t
	\right\}
	\le
	2\exp\left(
	-mL_m
	\min\left\{
	\frac98t^2,
	9t
	\right\}
	\right).
	\]
	Taking \(t=1/\sqrt m\) and using \(m\ge1\) gives
	\[
	\Pp_\pi
	\left\{
	\left|
	\norm{T_x}_F^2-1
	\right|>\frac1{\sqrt m}
	\right\}
	\le
	2e^{-L_m}
	\le
	\frac{\eta}{4N}.
	\]
	A union bound over \(X\) therefore yields \eqref{eq:main3} with probability at least $1-\eta/4$ over $\pi$. We denote this event by
	\[
	G=
	\left\{
	\max_{x\in X}
	\left|
	\norm{T_x}_F^2-1
	\right|
	\le
	\frac{1}{\sqrt m}
	\right\}.
	\]

	\medskip
	\noindent\emph{Step 4: Norm concentration via Talagrand's inequality.}
	Condition on the events \(E\), \(F\), and \(G\). The final step is to
	show that, outside an event of probability at most \(\eta/4\) over
	\(\xi\),
	\begin{equation}\label{eq:main4}
		\max_{x\in X}
		\left|
		\norm{\Phi x}_2-1
		\right|
		\le
		10\sqrt{\frac{\ell}{m}}.
	\end{equation}
	
	Recall that \(\Phi x=T_x\xi\). For fixed \(x\in X\), consider the
	convex function
	\[
	f_x(u)=\norm{T_xu}_2.
	\]
	Its Lipschitz constant is \(\norm{T_x}\). Let \(M_x\geq0\) be a median of
	\(f_x(\xi)\). By \Cref{lem:talagrand},
	\[
	\Pp_\xi
	\left\{
	\left|
	\norm{\Phi x}_2-M_x
	\right|\ge t
	\right\}
	\le
	4\exp\left(
	-\frac{t^2}{8\norm{T_x}^2}
	\right)
	\qquad\text{for }t\ge0.
	\]
	Moreover, integrating the preceding tail bound gives
	\[
		\E_\xi
		\left|
		\norm{\Phi x}_2-M_x
		\right|^2
		=
		\int_0^\infty
		2t\,
		\Pp_\xi
		\left\{
		\left|
		\norm{\Phi x}_2-M_x
		\right|\ge t
		\right\}
		\,dt
		\le
		8\int_0^\infty
		t\exp\left(
		-\frac{t^2}{8\norm{T_x}^2}
		\right)\,dt
		=
		32\norm{T_x}^2.
	\]
	On the other hand,
	\[
	\E_\xi\norm{\Phi x}_2^2
	=
	\E_\xi\norm{T_x\xi}_2^2
	=
	\norm{T_x}_F^2.
	\]
	Therefore, by the reverse triangle inequality in \(L^2\),
	\[
		\left|
		M_x-\norm{T_x}_F
		\right|
		=
		\left|
		\norm{M_x}_{L^2(\xi)}
		-
		\norm{\norm{\Phi x}_2}_{L^2(\xi)}
		\right|
		\le
		\left\|
		\norm{\Phi x}_2-M_x
		\right\|_{L^2(\xi)}
		\le
		4\sqrt2\,\norm{T_x}.
	\]
	Since on the event \(F\) we have
	\(\norm{T_x}^2\le2/m\), it follows that
	\[
	\left|
	M_x-\norm{T_x}_F
	\right|
	\le
	4\sqrt2\sqrt{\frac2m}
	=
	\frac8{\sqrt m}.
	\]
	Consequently, for every \(t\ge0\),
	\[
	\Pp_\xi
	\left\{
	\left|
	\norm{\Phi x}_2-\norm{T_x}_F
	\right|
	>
	\frac8{\sqrt m}+t
	\right\}\le
	4\exp\left(
	-\frac{t^2}{8\norm{T_x}^2}
	\right)
	\le
	4e^{-mt^2/16}.
	\]
	Taking $t=2\sqrt{5}\sqrt{\ell/m}$ and using \(\ell\ge\log 16\), we obtain
	\[
	4e^{-20\ell/16}
	\le
	2e^{-\ell}
	=
	\frac{\eta}{4N}.
	\]
	Therefore, a union bound over \(X\), together with \(\ell\ge \log 16\), yields
	\begin{equation}\label{eq:4.2.4.1}
		\max_{x\in X}
		\left|
		\norm{\Phi x}_2-\norm{T_x}_F
		\right|
		\le
		\left (\frac{8}{\sqrt{\log 16}} + 2\sqrt{5}\right )\sqrt{\frac{\ell}{m}}
	\end{equation}
	outside an event of probability at most \(\eta/4\) over \(\xi\).
	
	Recall that on the event \(G\),
	\[
	\left|
	\norm{T_x}_F-1
	\right|
	=
	\frac{
		\left|\norm{T_x}_F^2-1\right|
	}{
		\norm{T_x}_F+1
	}
	\le
	\frac{1}{\sqrt{m}}.
	\]
	Combining this estimate with \eqref{eq:4.2.4.1} and using \(\ell\geq \log 16\),
	we obtain
	\[
	\max_{x\in X}
	\left|
	\norm{\Phi x}_2-1
	\right|
	\le
	\left (\frac{9}{\sqrt{\log 16}} + 2\sqrt{5}\right )\sqrt{\frac{\ell}{m}}
	\leq
	10\sqrt{\frac{\ell}{m}}.
	\]
	This proves \eqref{eq:main4}. Since the exceptional probabilities in
	Steps \(1\)--\(4\) sum to at most \(\eta\), the proof is complete.
\end{proof}

We conclude this section with a sharpening of \Cref{thm:finite-JL} for the
standard basis
\[
X=\{e_1,\ldots,e_d\}.
\]
In this case, for every \(j\in[d]\), the probability vector
\[
q=(Be_j)^2
\]
satisfies \(\norm q_\infty\le\norm B_\infty^2\) deterministically.
Thus, the inner-sign flattening step in the proof of
\Cref{thm:finite-JL} is unnecessary, and the corresponding logarithmic
loss is avoided. Repeating the remaining steps yields the following
amplification principle for the sampled column energy.

\begin{corollary}\label{cor:leverage-amplification}
	Let \(A,B\in O(d)\), let \(I\subseteq [d]\) with \(|I|=m\), let
	\(D_\xi  \in \mathbb{R}^{d\times d}\) be a Rademacher diagonal, let \(\pi\) be an
	independent uniform permutation on \(S_d\), and let \(0<\eta<1/2\). Set
	\[
	U=AD_\xi P_\pi B, \qquad
	\ell=\log\left(\frac{8d}{\eta}\right),
	\qquad
	L_m=\log\left(\frac{8dm}{\eta}\right).
	\]
	Suppose that $m\ge600\ell$ and
	\[
	36\kappa_I(A)\norm B_\infty^2\,mL_m\le1.
	\]
	Then, with probability at least \(1-\eta\),
	\[
	\kappa_I(U)\le2.
	\]
\end{corollary}

\begin{proof}
	Apply Steps \(2\)--\(4\) of the proof of
	\Cref{thm:finite-JL} to the standard basis
	\(\{e_1,\ldots,e_d\}\). For each \(j\in[d]\),
	\[
	q=(Be_j)^2
	\qquad\text{satisfies}\qquad
	\norm q_\infty\le\norm B_\infty^2,
	\]
	so Step \(1\) is unnecessary. Repeating Steps \(2\) and \(3\) with
	this deterministic bound, followed by Step \(4\) and a union bound
	over \(j\in[d]\), gives
	\[
	\max_{j\in[d]}
	\left|
	\sqrt{\frac dm}\norm{P_IUe_j}_2-1
	\right|
	\le
	10\sqrt{\frac{\ell}{m}}
	\]
	outside an event of probability at most \(3\eta/4\).
	Since \(m\ge600\ell\),
	\[
	10\sqrt{\frac{\ell}{m}}
	\le\frac{10}{\sqrt{600}}
	<\sqrt2-1.
	\]
	Hence
	\[
	\max_{j\in[d]}
	\frac dm\norm{P_IUe_j}_2^2
	\le2.
	\]
	By the definition of \(\kappa_I(U)\), this is precisely the desired
	conclusion.
\end{proof}

\section{Extensions to structured infinite sets}
\label{sec:infinite-sets}

The goal of this section is to extend \Cref{thm:finite-JL} from finite
sets to several infinite models of interest, including sparse vectors,
finite unions of subspaces, and low-rank matrices. The strategy is to find a finite
family of directions such that controlling every quadratic form
\(x\mapsto\langle Hx,x\rangle\) on this family controls it uniformly on
the entire model. Arguments of this type are implicit in the standard
net proofs for sparse and low-rank restricted isometries
\cite{BaraniukEtAl2008,RechtFazelParrilo2010}, while norming sets are
classical objects in approximation theory \cite{BrudnyiYomdin2015}.
More recently, polynomial norming sets were used in
\cite{ZhangKileel2025} to transfer finite-point sketching guarantees to algebraic
varieties and polynomial images.

\subsection{Quadratic norming families}

\begin{definition}\label{def:quadratic-norming}
	Let \(X\subseteq \mathbb S^{d-1}\) and $\lambda\geq1$. A finite set
	\(W\subseteq \mathbb S^{d-1}\) is a
	\emph{$\lambda$-quadratic norming family} for \(X\) if, for every
	symmetric matrix \(H\in\R^{d\times d}\),
	\[
	\sup_{x\in X}|\langle Hx,x\rangle|
	\le
	\lambda
	\sup_{w\in W}|\langle Hw,w\rangle|.
	\]
\end{definition}

By a simple convexity argument, one can see that this is equivalent to the inclusion
\[
\{xx^{\mathsf T}:x\in X\}
\subseteq 
\lambda\,
\operatorname{co}
\left(
\pm\{ww^{\mathsf T}:w\in W\}
\right),
\]
where \(\operatorname{co}\) denotes the convex hull. The following proposition makes the transfer principle explicit: if
\(W\) is a \(\lambda\)-quadratic norming family for \(X\), then uniform
control of the squared-norm distortion on \(W\) extends to \(X\), at
the cost of a factor \(\lambda\).

\begin{proposition}\label{prop:norming-transfer}
	Let $W\subseteq \mathbb{S}^{d-1}$ be a
	finite \(\lambda\)-quadratic norming family for a set \(X\subseteq \mathbb{S}^{d-1}\). Then, for
	every linear map \(\Phi:\R^d\to\R^m\),
	\[
	\sup_{x\in X}
	\left|
	\norm{\Phi x}_2^2-1
	\right|
	\le
	\lambda
	\sup_{w\in W}
	\left|
	\norm{\Phi w}_2^2-1
	\right|.
	\]
\end{proposition}

\begin{proof}
	Set \(H=\Phi^{\mathsf T}\Phi-I_d\). By the quadratic norming
	property,
	\[
	\begin{aligned}
		\sup_{x\in X}
		\left|
		\norm{\Phi x}_2^2-1
		\right|
		&=
		\sup_{x\in X}
		|\langle Hx,x\rangle|
		\le
		\lambda
		\sup_{w\in W}
		|\langle Hw,w\rangle|
		=
		\lambda
		\sup_{w\in W}
		\left|
		\norm{\Phi w}_2^2-1
		\right|.
	\end{aligned}\qedhere
	\]
\end{proof}

In other words, any linear map that is a Johnson--Lindenstrauss
embedding for \(W\) is automatically one for \(X\), with distortion
increased by at most the factor \(\lambda\).

Although elementary, this transfer principle yields consequences for a
variety of infinite models. The next step, and the purpose of the following subsection, is to identify conditions under which a finite quadratic norming family exists.

\subsection{Secant-generating models}

Let \(\mathcal C\subseteq \R^d\) be a nonzero balanced cone, that is,
\(tx\in\mathcal C\) whenever \(x\in\mathcal C\) and \(t\in\R\). Set
\(S_{\mathcal C}:=\mathcal C\cap\mathbb S^{d-1}\), and define its
normalized secant set by
\[
\seca(S_{\mathcal C})
:=
\left\{
\frac{x-y}{\norm{x-y}_2}:
x,y\in S_{\mathcal C},\ x\neq y
\right\}.
\]

\begin{definition}\label{def:secant-generating}
	Let \( \lambda \ge1\). We say that \(\mathcal C\) is
	\emph{\(\lambda\)-secant-generating} if, for every
	\(x,y\in S_{\mathcal C}\), there exist \(s\in\mathbb N\) and
	\(h_1,\ldots,h_s\in\mathcal C\) such that
	\[
	x-y
	=
	\sum_{j=1}^s h_j
	\qquad\text{and}\qquad
	\sum_{j=1}^s\norm{h_j}_2
	\le
	 \lambda \norm{x-y}_2.
	\]
	Equivalently,
	\[
	\seca (S_{\mathcal C})
	\subseteq 
	\lambda \operatorname{co}(S_{\mathcal C}).
	\]
\end{definition}

For \(X\subseteq \mathbb S^{d-1}\) and \(\delta>0\), we denote by
\[N(X,\delta)\] 
the smallest cardinality of a Euclidean
\(\delta\)-net of \(X\).

The next result shows that, for a secant-generating model, a net can be converted into a finite quadratic norming family of comparable logarithmic size.

\begin{proposition}\label{prop:secant-norming}
	Let \(\mathcal C\subseteq \R^d\) be a
	\(\lambda\)-secant-generating balanced cone. Then
	\(S_{\mathcal C}=\mathcal C\cap \mathbb{S}^{d-1}\) admits a \(2\)-quadratic norming family
	\(W\subseteq \seca (S_{\mathcal C})\) satisfying
	\[
	|W|
	\le
	2N\left(S_{\mathcal C},\frac1{4\lambda}\right)^2.
	\]
\end{proposition}

\begin{proof}
	Let \(\mathcal N\subseteq S_{\mathcal C}\) be a \((4\lambda)^{-1}\)-net of
	\(S_{\mathcal C}\) with minimal cardinality. Define the family
	\[
	W
	=
	\left\{
	\frac{u+v}{\norm{u+v}_2}:
	u,v\in\mathcal N,\ u\neq -v
	\right\}
	\cup
	\left\{
	\frac{u-v}{\norm{u-v}_2}:
	u,v\in\mathcal N,\ u\neq v
	\right\}.
	\]
	Since \(\mathcal C\) is balanced,
	\(u+v=u-(-v)\), and hence
	\(W\subseteq \seca (S_{\mathcal C})\). Moreover,
	\[
	|W|\le2|\mathcal N|^2.
	\]
	We will show that $W$ is \(2\)-quadratic norming for $S_{\mathcal C}$. Let \(H\) be symmetric and set
	\[
	\alpha
	=
	\sup_{w\in W}|\langle Hw,w\rangle|.
	\]
	For \(u,v\in\mathcal N\), polarization gives
	\[
	4\langle Hu,v\rangle
	=
	\langle H(u+v),u+v\rangle
	-
	\langle H(u-v),u-v\rangle.
	\]
	It follows that
	\begin{equation}\label{eq:5.1}
	|\langle Hu,v\rangle|
	\le
	\frac{\alpha}{4}
	\left(
	\norm{u+v}_2^2+\norm{u-v}_2^2
	\right)
	= \frac{\alpha}{2} (\|u\|_2^2 + \|v\|_2^2)=
	\alpha.
	\end{equation}
	Thus, control of quadratic forms in $W$ gives control of bilinear forms in $\mathcal{N}$. Set
	\[
	M
	=
	\sup_{x,y\in S_{\mathcal C}}
	|\langle Hx,y\rangle|.
	\]
	Fix \(x,y\in S_{\mathcal C}\), and choose \(u,v\in\mathcal N\)
	such that
	\[
	\norm{x-u}_2,\norm{y-v}_2
	\le
	\frac1{4\lambda}.
	\]
	By the triangle inequality,
	\begin{equation}\label{eq:5.2}
	|\langle Hx,y\rangle|
	\le
	|\langle Hu,v\rangle|
	+
	|\langle H(x-u),y\rangle|
	+
	|\langle Hu,y-v\rangle|
	\end{equation}
	The first term in the right hand side is bounded in (\ref{eq:5.1}) by $\alpha$. Also, since \(\mathcal C\) is \(\lambda\)-secant-generating, there are $h_1,\ldots,h_s \in \mathcal{C}$ such that
	\[
	|\langle H(x-u),y\rangle| \leq  \sum_{j=1}^s |\langle Hh_j,y \rangle | \leq M \sum_{j=1}^s \|h_j\|_2
	\le
	\lambda M\norm{x-u}_2
	\le
	\frac14M.
	\]
	Similarly
	\[
	|\langle Hu,y-v\rangle|
	\le
	\frac14M.
	\]
	Therefore, from (\ref{eq:5.2}) we deduce that
	\[
	|\langle Hx,y\rangle|
	\le
	\alpha+\frac12M.
	\]
	Taking the supremum gives \(M\le2\alpha\). In particular,
	\[
	\sup_{x\in S_{\mathcal C}}|\langle Hx,x\rangle|
	\le
	2\sup_{w\in W}|\langle Hw,w\rangle|,
	\]
	which completes the proof.
\end{proof}

Recall that the Gaussian width of \(X\subseteq\mathbb S^{d-1}\) is
\[
w(X)
=
\E\sup_{x\in X}\langle g,x\rangle,
\]
where $g\sim N(0,I_d)$ is a standard Gaussian vector in $\mathbb{R}^d$. By Sudakov's inequality \cite[Corollary~7.4.2]{Vershynin2026},
\[
\log N(X,\delta)
\le
C\delta^{-2}w(X)^2.
\]
Together with \Cref{prop:secant-norming}, this yields a finite
\(2\)-quadratic norming family
\(W\subseteq\seca(S_{\mathcal C})\) with
\[
\log|W|
\le
C\lambda^2w(S_{\mathcal C})^2.
\]
Thus, \Cref{thm:finite-JL} and \Cref{prop:norming-transfer} allow us
to pass from the finite family \(W\) to the full model
\(S_{\mathcal C}\).

\begin{corollary}\label{cor:gaussian-width}
	Let \(\mathcal C\subseteq\R^d\) be a
	\(\lambda\)-secant-generating balanced cone, let
	\(0<\eps<1\) and \(0<\eta<1/2\), and set
	\[
	a_{\mathcal C}
	=
	C_0\lambda^2w(S_{\mathcal C})^2
	+
	\log\left(\frac{16}{\eta}\right),
	\]
	where \(C_0>0\) is a universal constant. Assume that $m\geq2025\eps^{-2}a_{\mathcal C}$ and
	\[
	36\kappa_I(A)\Theta_B(\seca(S_{\mathcal C}))\,m
	\left(a_{\mathcal C}+\log d\right)
	\left(a_{\mathcal C}+\log m\right)
	\le d.
	\]
	Then, with probability at least \(1-\eta\),
	\[
	\left|
	\norm{\Phi x}_2^2-\norm{x}_2^2
	\right|
	\le
	\eps\norm{x}_2^2
	\qquad
	\text{for every }x\in\mathcal C.
	\]
\end{corollary}

\begin{proof}
	By Sudakov's inequality and \Cref{prop:secant-norming}, there is a
	\(2\)-quadratic norming family
	\(W\subseteq \seca(S_{\mathcal C})\) such that
	\[
	\log\left(\frac{8|W|}{\eta}\right)
	\le a_{\mathcal C}.
	\]
	Moreover,
	\(\Theta_B(W)\le\Theta_B(\seca(S_{\mathcal C}))\).
	Hence, \Cref{thm:finite-JL} gives
	\[
	\sup_{w\in W}
	\left|\norm{\Phi w}_2-1\right|
	\le
	10\sqrt{\frac{a_{\mathcal C}}m}
	\le\frac{2\eps}{9}.
	\]
	Thus,
	\[
	\sup_{w\in W}
	\left|\norm{\Phi w}_2^2-1\right| = \sup_{w\in W} \left |\norm{\Phi w}_2-1\right| \left |\norm{\Phi w}_2+1\right|
	\leq \left (2+\frac{2\varepsilon}{9}\right ) \sup_{w\in W} \left |\norm{\Phi w}_2-1\right|
	\le\frac{40}{81}\eps.
	\]
	The conclusion follows from \Cref{prop:norming-transfer} and
	homogeneity.
\end{proof}

In particular, when \(\lambda\) is universally bounded and \(\eta\) is
fixed, the number of rows can be taken to be optimal:
\[
m
\asymp
\eps^{-2}w(S_{\mathcal C})^2.
\]
This, of course, is subject to the admissibility condition above.

\subsection{Applications}

The preceding results reduce the extension to an infinite model to a
geometric problem: constructing a finite quadratic norming family and
estimating its cardinality. We now carry out this program for
several model classes of interest.

\medskip
\noindent\emph{Sparse vectors and the restricted isometry property.}
For \(1\le s\le d\), consider the set of $s$-sparse vectors
\[
\Sigma_s
:=
\left\{
x\in\R^d:\norm{x}_0\le s
\right\}.
\]
This set is clearly a balanced cone. Moreover, it is \(\sqrt2\)-secant-generating. Indeed, if
\(x,y\in S_{\Sigma_s}\), then \(x-y\) is \(2s\)-sparse. Splitting its
support into two sets of cardinality at most \(s\), we may write
\[
x-y=h_1+h_2,
\qquad
h_1,h_2\in\Sigma_s,
\qquad
\langle h_1,h_2\rangle=0,
\]
and thus
\[
\norm{h_1}_2+\norm{h_2}_2
\le
\sqrt2\norm{x-y}_2.
\]
We may therefore apply Corollary \ref{cor:gaussian-width} once we estimate the Gaussian width of $S_{\Sigma_s}$ and the maximal variance on its normalized secants. The standard Gaussian-width estimate
\[
w(S_{\Sigma_s})^2
\asymp
s\log\left(\frac{ed}{s}\right)
\]
is classical; see, for example, \cite[Exercise~9.27]{Vershynin2026}. Moreover,
\[
\Theta_B(\seca (S_{\Sigma_s}))
=
d\|B\|_\infty^2.
\]
Indeed, the upper bound follows from the definition, while
the reverse inequality follows because every coordinate vector belongs
to \(\seca (S_{\Sigma_s})\). 
Recall that a linear map \(\Phi:\R^d\to\R^m\) has restricted isometry
constant of order \(s\) given by
\[
\delta_s(\Phi)
=
\sup_{x\in S_{\Sigma_s}}
\left|
\norm{\Phi x}_2^2-1
\right|.
\]
Corollary \ref{cor:gaussian-width} and the above estimates yield the following.

\begin{corollary}\label{cor:sparse-RIP}
	Let \(0<\eps<1\) and \(0<\eta<1/2\), and set
	\[
	a_s
	=
	Cs\log\left(\frac{ed}{s}\right)
	+
	\log\left(\frac{16}{\eta}\right),
	\]
	where \(C>0\) is a universal constant. Assume that $m\ge2025\eps^{-2}a_s$ and
	\[
	144\kappa_I(A)\norm{B}_\infty^2\,ma_s^2
	\le1.
	\]
	Then, with probability at least \(1-\eta\), the linear map
	\(\Phi:\mathbb R^d\to\mathbb R^m\) defined in
	\eqref{eq:Phi-main} satisfies
	\[
	\delta_s(\Phi)\le\eps.
	\]
\end{corollary}

\begin{proof}
	The result follows from the preceding observations and
	\Cref{cor:gaussian-width}. Indeed,
	\[
	\Theta_B(\seca(S_{\Sigma_s}))
	=
	d\norm B_\infty^2,
	\]
	and \(a_s\ge\log d\ge\log m\) for $C$ large enough. Hence
	\[
	\left(a_s+\log d\right)
	\left(a_s+\log m\right)
	\le4a_s^2,
	\]
	so the assumed bound implies the admissibility condition in
	\Cref{cor:gaussian-width}.
\end{proof}

Thus, for fixed $\eta$, throughout the admissible range, $\Phi$ achieves the restricted isometry property with high probability at the optimal row count
\[
m
\asymp
\eps^{-2}
\left(
s\log\frac{ed}{s}
\right).
\]

\medskip
\noindent\emph{Low-rank matrices.}
Let \(d=n_1n_2\), and identify \(\R^d\) with
\(\R^{n_1\times n_2}\). For
\(1\le r\le\min\{n_1,n_2\}\), let
\[
\mathcal R_r
:=
\left\{
X\in\R^{n_1\times n_2}:
\operatorname{rank}(X)\le r
\right\}.
\]
The balanced cone \(\mathcal R_r\) is
\(\sqrt2\)-secant-generating. Indeed, if
\(X,Y\in S_{\mathcal R_r}\), then \(X-Y\) has rank at most \(2r\).
Splitting its singular value decomposition into two orthogonal parts
of rank at most \(r\), we may write
\[
X-Y=H_1+H_2,
\qquad
H_1,H_2\in\mathcal R_r,
\qquad
\langle H_1,H_2\rangle=0,
\]
and hence
\[
\norm{H_1}_F+\norm{H_2}_F
\le
\sqrt2\norm{X-Y}_F.
\]
The standard Gaussian-width estimate
\[
w(S_{\mathcal R_r})^2
\asymp
r(n_1+n_2-r)
\]
follows, for example, from
\cite[Exercise~4.50]{Vershynin2026}
together with Sudakov's and Dudley's inequalities. Moreover,
\[
\Theta_B(\seca(S_{\mathcal R_r}))
=
d\norm B_\infty^2,
\]
since every coordinate basis matrix belongs to
\(\seca(S_{\mathcal R_r})\).

Thus, the same application of \Cref{cor:gaussian-width} as in the
sparse case yields the rank restricted isometry property with high
probability throughout the corresponding admissible range, using the
optimal row count
\[
m
\asymp
\eps^{-2}r(n_1+n_2-r)
\]
for fixed \(\eta\). We omit the corresponding formal statement to avoid repetition.

\medskip
\noindent\emph{Finite unions of subspaces.}
Let $E_1,\ldots,E_s$ be subspaces of $\mathbb{R}^d$ and let 
\[ \mathscr E:=\bigcup_{j=1}^s E_j. \]
Although this set is a balanced cone,
its secants generally belong to sums of two subspaces. A more direct
approach is to use the fact that nets are quadratic norming families for a subspace. Consequently, the union of
such nets is a quadratic norming family for \(S_{\mathscr E}=\mathscr E \cap \mathbb{S}^{d-1}\). Set
\[
a_{\mathscr E}
=
\log\left(
\frac8\eta
\sum_{j=1}^s 9^{\dim E_j}
\right).
\]

\begin{proposition}\label{prop:subspace-family}
	Let \(0<\eps<1\) and \(0<\eta<1/2\). Assume that $m\geq 2025\eps^{-2}a_{\mathscr E}$ and
	\[
	36\kappa_I(A)\Theta_B(S_{\mathscr E})\,m
	\left(a_{\mathscr E}+\log d\right)
	\left(a_{\mathscr E}+\log m\right)
	\le
	d.
	\]
	Then, with probability at least \(1-\eta\), the random map
	\(\Phi:\R^d\to\R^m\) defined in \eqref{eq:Phi-main} satisfies
	\[
	\left|
	\norm{\Phi x}_2^2-\norm{x}_2^2
	\right|
	\le
	\eps\norm{x}_2^2
	\qquad
	\text{for every }x\in\mathscr E.
	\]
\end{proposition}

\begin{proof}
	For each \(j\in[s]\), let
	\(\mathcal N_j\subseteq E_j\cap\mathbb S^{d-1}\) be a \(1/4\)-net of $E_j \cap \mathbb{S}^{d-1}$ with $|\mathcal N_j|\le9^{\dim E_j}$, and set
	\[
	W
	=
	\bigcup_{j=1}^s\mathcal N_j.
	\]
	We first show that \(W\) is a \(2\)-quadratic norming family for
	\(S_{\mathscr E}\). Let \(H\) be symmetric, fix \(j\in[s]\), and set
	\[
	M_j
	=
	\sup_{x\in E_j\cap\mathbb S^{d-1}}
	|\langle Hx,x\rangle|.
	\]
	For \(x\in E_j\cap\mathbb S^{d-1}\), choose
	\(u\in\mathcal N_j\) such that \(\norm{x-u}_2\le1/4\). Since
	\(M_j\) is the operator norm of the restriction of \(H\) to
	\(E_j\),
	\[
	\begin{aligned}
		|\langle Hx,x\rangle|
		&\le
		|\langle Hu,u\rangle|
		+
		|\langle H(x-u),x\rangle|
		+
		|\langle Hu,x-u\rangle|\\
		&\le
		|\langle Hu,u\rangle|
		+
		\frac12M_j.
	\end{aligned}
	\]
	Taking the supremum over \(x\) gives
	\[
	M_j
	\le
	2\sup_{u\in\mathcal N_j}|\langle Hu,u\rangle|.
	\]
	Hence \(W\) is \(2\)-quadratic norming for \(S_{\mathscr E}\). Observe that
	\[
	\log\left(\frac{8|W|}{\eta}\right)
	\le
	\log\left(\frac{8}{\eta}\sum_{j=1}^s9^{\dim E_j}\right)
	=
	a_{\mathscr E},
	\]
	and also
	\[
	\Theta_B(W)\le\Theta_B(S_{\mathscr E}).
	\]
	The assumptions thus imply the admissibility condition in
	\Cref{thm:finite-JL}, and hence
	\[
	\sup_{w\in W}
	\left|
	\norm{\Phi w}_2-1
	\right|
	\le
	10\sqrt{\frac{a_{\mathscr E}}{m}}
	\le
	\frac{2\eps}{9}.
	\]
	Since \(\eps<1\),
	\[
	\sup_{w\in W}
	\left|
	\norm{\Phi w}_2^2-1
	\right|
	\le
	\frac{2\eps}{9}
	\left(2+\frac{2\eps}{9}\right)
	\le
	\frac{40}{81}\eps
	<
	\frac{\eps}{2}.
	\]
	The conclusion follows from \Cref{prop:norming-transfer} and
	homogeneity.
\end{proof}

Many structured models fall naturally into this framework, including
block, group, and tree sparsity, and jointly sparse matrices.
Such models can be represented as finite unions of
low-dimensional subspaces, so \Cref{prop:subspace-family} yields
restricted-isometry estimates governed by the largest subspace
dimension and the logarithm of the number of admissible subspaces.
Since the derivations are analogous, we omit the individual
statements.

\section{Concrete fast orthogonal constructions}
\label{sec:mixers}

The results of the previous sections are deterministic in the
orthogonal matrices \(A\) and \(B\). They may therefore be applied
conditionally to random matrices, provided that they are
independent of the random signs and permutation appearing in
\eqref{eq:Phi-main}. Thus, once high-probability estimates for
\(\kappa_I(A)\) and \(\Theta_B(X)\) are available, the Johnson--Lindenstrauss and restricted isometry guarantees from the previous sections follow by conditioning and adding the failure
probabilities. 

A useful, and perhaps unexpected, point is that \(A\) and \(B\) need
not be independent of each other. In particular, one may take \(A=B\).

In this section, we examine several concrete orthogonal transforms and
random orthogonal ensembles from the literature. We estimate
their parameters \(\kappa_I(A)\) and \(\Theta_B(X)\), and draw conclusions from the results developed in the previous sections.

Recall that, for every \(U\in O(d)\), every \(I\subseteq [d]\), and every
\(X\subseteq \mathbb S^{d-1}\),
\begin{equation}\label{eq:local-parameters-from-coherence}
	1\leq\kappa_I(U)
	\le
	d\norm U_\infty^2,
	\qquad
	1\leq\Theta_U(X)
	\le
	d\norm U_\infty^2.
\end{equation}
Thus, coherence provides a general upper bound for both parameters.
For several of the random ensembles considered below, however, direct
estimates of the sampled column energy or maximal variance are
substantially sharper than this global coherence bound.

\subsection{Flat deterministic transforms}

We begin with two standard deterministic transforms. Let \(H\) be a
normalized Walsh--Hadamard matrix, available when \(d\) is a power of
two. Since
\[
H_{ij}^2=\frac1d,
\]
from \eqref{eq:local-parameters-from-coherence} it follows that
\begin{equation}\label{eq:hadamard-parameters}
	\kappa_I(H)=1,
	\qquad
	\Theta_H(X)=1
\end{equation}
for every \(I\subseteq [d]\) and every
\(X\subseteq \mathbb S^{d-1}\).

We may also consider the orthonormal discrete cosine transform
\(C\in\mathbb R^{d\times d}\), which is available for every \(d\) and
satisfies
\[
\norm C_\infty^2
\le
\frac2d.
\]
Therefore
\begin{equation}\label{eq:dct-parameters}
	\kappa_I(C)\le2,
	\qquad
	\Theta_C(X)\le2.
\end{equation}
Both transforms can be applied in \(O(d\log d)\) time using fast
transform algorithms.

These estimates already yield concrete fast Johnson--Lindenstrauss
embeddings for finite sets. Indeed, suppose that \(A\) and \(B\) are chosen to be
either of the two transforms above, and consider
\[
\Phi
=
\sqrt{\frac dm}\,P_IAD_\xi P_\pi BD_{\xi'}.
\]
Since both \(\kappa_I(A)\) and \(\Theta_B(X)\) are bounded by universal
constants, \Cref{thm:finite-JL} shows that \(\Phi\) is a
Johnson--Lindenstrauss embedding with the optimal row count
\[
m
\asymp
\eps^{-2}\log\left(\frac{N}{\eta}\right),
\]
provided that
\[
m
\log\left(\frac{Nd}{\eta}\right)
\log\left(\frac{Nm}{\eta}\right)
\lesssim d.
\]
In particular, for fixed \(\eps\) and \(\eta\), this allows
\(\log N\lesssim d^{1/3}\), up to logarithmic considerations in the
transition between the different regimes. The same estimates also yield, through the results of
Section~\ref{sec:infinite-sets}, analogous guarantees for sparse vectors, low-rank matrices, and finite unions of subspaces.

This cube-root range reflects the limitations of the present
local parameter framework rather than those of fast structured
orthogonal transforms themselves. More refined analyses can exploit
additional structure of the underlying construction. For example,
Ailon and Liberty \cite{AilonLiberty2009} obtain similar fast
Johnson--Lindenstrauss transforms in the substantially larger
\(d^{1/2-\delta}\) regime, for any fixed \(\delta>0\), using a more
detailed analysis of the structured transform. Our approach instead
expresses the embedding guarantee entirely through
\(\kappa_I(A)\) and \(\Theta_B(X)\). As shown in Section~7, the
resulting cube-root barrier cannot, in general, be overcome using only
these two parameters.

\subsection{The uniform angle Kac walk}

A Kac update selects two distinct coordinates uniformly and applies a
rotation by an independent uniform angle in their coordinate plane. A
Kac walk of length $T$ is the product of \(T\) independent updates. This Markov chain was recently used in \cite{JainEtAl2022}
to construct fast and memory-efficient Johnson--Lindenstrauss embeddings
for finite sets. 

The following result collects key estimates from \cite[Lemma~3.3]{JainEtAl2022} in the form needed here.

\begin{proposition}\label{prop:kac-parameters}
	Let \(Q_T\) be a Kac walk of length $T$, and let
	\(I\subseteq [d]\), \(|I|=m\), be fixed. Let \(0<\varepsilon<1/2\) and
	\(0<\eta<1/2\). Suppose
	\[
	T
	\ge
	10d\log\left(\frac{2d}\eta\right)
	\qquad \text{and}
	\qquad
	m
	\ge
	64\varepsilon^{-2}\log\left(\frac{8d}\eta\right).
	\]
	Then, with probability at least \(1-\eta\),
	\[
	\kappa_I(Q_T)
	\le
	1+\varepsilon
	\]
	and
	\[
	\norm{Q_T}_\infty
	\le
	3\sqrt{\frac{\log(d/\eta)}{d}}.
	\]
\end{proposition}

\begin{proof}
	We first use the following estimate from
	\cite[Lemma~3.3]{JainEtAl2022}: if \(X_T=Q_TX_0\), then for
	\(1\le k\le d\) and \(0<\varepsilon<1/2\),
	\[
	\Pp\left\{
	\sum_{i=1}^k (X_T)_i^2
	\notin
	\frac{k}{d}[1-\varepsilon,1+\varepsilon]
	\right\}
	\le
	8d^4e^{-T/(2d)}
	+
	2e^{-\varepsilon^2k/64}.
	\]
	By coordinate permutation symmetry, the same estimate holds with the
	first \(m\) coordinates replaced by the fixed set \(I\). Apply this with \(X_0=e_j\) and \(k=m\).  Thus, a
	union bound over \(j\in[d]\) gives
	\[
	\Pp\left\{\kappa_I(Q_T)>1+\varepsilon\right\}
	\le
	8d^5e^{-T/(2d)}
	+
	2d e^{-\varepsilon^2m/64}
	\le
	\frac{\eta}{2}.
	\]
	
	For the second conclusion, apply the second part of
	\cite[Lemma~3.3]{JainEtAl2022} with
	\[
	K
	=
	3\sqrt{\frac{\log(d/\eta)}{\log d}}.
	\]
	Using $\eta\leq 1/2$, this yields
	\[
	\Pp\left\{
	\norm{Q_T}_\infty
	>
	3\sqrt{\frac{\log(d/\eta)}{d}}
	\right\}
	\le
	2d^3e^{-T/(2d)}
	+
	2d^{5/2}
	\left(\frac{\eta}{d}\right)^{9/2}
	\le
	\frac{\eta}{2}.
	\]
	A union bound completes the proof.
\end{proof}

Thus, a Kac factor used on the left has essentially optimal sampled
column energy once \(m\gtrsim\log(d/\eta)\). On the right, the global
entrywise estimate gives
\[
\Theta_{Q_T}(X)
\le
d\norm{Q_T}_\infty^2
=O\left (\log \left(d/\eta\right )\right ).
\]
The global estimate for the right parameter is sharp whenever
\(X\) contains the standard basis, since
\[
\Theta_U(\{e_1,\ldots,e_d\})
=
d\norm U_\infty^2.
\]

The preceding proposition allows us to use Kac walks directly in our
general framework. Fix \(\eta\), and let \(A\) and \(B\) be independent
Kac walks of length \(T=O(d\log d)\). With high probability,
\[
\kappa_I(A)\Theta_B(X)
=
O(\log d).
\]
Thus, \Cref{thm:finite-JL} yields the optimal Johnson--Lindenstrauss
row count
\[
m\asymp\eps^{-2}\log N
\]
throughout the corresponding admissible range, which has a
logarithmic loss compared with the flat deterministic examples above.
The results of Section~\ref{sec:infinite-sets} similarly yield
restricted isometry guarantees for structured infinite models.

In this setting there is a useful simplification. The law of a Kac
walk is invariant under conjugation by signed permutation matrices.
Consequently, the middle signed permutation in our model
can be moved through $B$, after which the two
independent Kac walks combine into a single longer walk.

\begin{corollary}\label{cor:kac-JL}
	There exist universal constants \(c,C>0\) such that the following
	holds. Let \(X\subseteq \mathbb S^{d-1}\) with \(|X|=N\geq d\),
	\(0<\eps<1\), \(0<\eta<1/2\), and
	\(I\subseteq [d]\) with \(|I|=m\). Let \(Q_T\) be a Kac walk with
	\(T\geq Cd\log(d/\eta)\), let \(D_\xi\) be an independent
	Rademacher diagonal, and let \(\pi\) be an independent uniform
	permutation on \(S_d\). Define
	\[
	\Phi=\sqrt{\frac dm}\,P_IQ_TD_\xi P_\pi.
	\]
	Assume the admissibility condition
	\[
	m\log\left(\frac d\eta\right)
	\log\left(\frac{Nd}{\eta}\right)
	\log\left(\frac{Nm}{\eta}\right)
	\leq c d.
	\]
	If $m\geq C\eps^{-2}\log(N/\eta)$,
	then, with probability at least \(1-\eta\),
	\[
	\sup_{x\in X}
	\left|
	\norm{\Phi x}_2-\norm{x}_2
	\right|
	\le
	\eps\norm{x}_2.
	\]
\end{corollary}

\begin{proof}
	Split \(Q_T=AB\), where \(A\) and \(B\) are
	independent Kac walks of comparable length. Let
	\(\mathcal R=D_{\xi}P_\pi\) be a signed permutation matrix, and let $\xi'$ be an independent Rademacher vector. Since the law of \(B\) is
	invariant under conjugation by signed permutations,
	\[
	A\mathcal RBD_{\xi'}
	=
	A(\mathcal RB\mathcal R^{\mathsf T})(\mathcal RD_{\xi'})
	\stackrel{\mathrm d}{=}
	ABD_{\widetilde\xi}P_\pi
	=
	Q_TD_{\widetilde\xi}P_\pi,
	\]
	where \(\widetilde\xi\) is again a Rademacher vector. By \Cref{prop:kac-parameters}, applied with failure probability
	\(\eta/4\) to each of \(A\) and \(B\), and a union bound, with
	probability at least \(1-\eta/2\),
	\[
	\kappa_I(A)=O(1),
	\qquad
	\Theta_B(X)
	\le
	d\norm B_\infty^2
	=O\left(\log (d/\eta)\right).
	\]
	The assumed range for \(m\) therefore implies the admissibility
	condition in \Cref{thm:finite-JL}. Applying it with failure probability \(\eta/2\), the conclusion follows after adjusting the universal constants.
\end{proof}

This should be compared with the more refined analysis in \cite{JainEtAl2022}. Their construction uses a two-stage scheme that is
analyzed through chaining and restricted isometry arguments, rather than only the
parameters \(\kappa_I(A)\) and \(\Theta_B(X)\). This allows them to
reach, up to logarithmic factors, a square root regime
\[
\log N= O(\eps\sqrt d)
\]
while retaining \(O(d\log d)\) application time in that range. By contrast, the
preceding corollary shows that, within our smaller admissible range, a
single Kac walk already suffices. Our approach is more direct and fits
into the general framework developed above, but consequently inherits
the cube-root limitation. As discussed
in Section~7, overcoming this limitation requires additional structural information.

\subsection{ORA and symmetric ORA}

Orthogonal repeated averaging, abbreviated ORA, is the fixed-angle analogue of the Kac walk, in which every update uses the angle $\pi/4$. More precisely, at each ORA update, an ordered pair $(i,j)$ with $i\neq j$ is chosen uniformly from the $d(d-1)$ ordered pairs of distinct coordinates, and a rotation by $\pi/4$ is applied to $(x_i,x_j)$. The remaining coordinates are unchanged. Following
\cite{JainEtAl2022}, symmetric ORA, abbreviated S-ORA, uses instead an
angle chosen uniformly from
\[
\left\{
\frac\pi4,\frac{3\pi}4,\frac{5\pi}4,\frac{7\pi}4
\right\}.
\]

We now turn to the proof of \Cref{thm:one-block-ora} and establish,
in addition, a constant sampled column energy estimate for ORA
throughout a corresponding admissible range.

Our argument has three ingredients. First, fixed-order moment estimates
give a weak bound on the largest entry of a short ORA walk. Second, the coherence and sampled column energy amplification results turn these preliminary bounds into near optimal control of the maximum entry and sampled column energy. Finally, the symmetries of
S-ORA identify the auxiliary two-block product with one longer walk,
and an exact coupling transfers the resulting parameter estimates back
to ordinary ORA.

We begin with a preliminary coherence estimate for a short ORA walk, obtained from the fixed-order moment bounds in \cite{JainEtAl2022}. Although this estimate is
far from the optimal coherence scale, it is strong enough to initiate
the amplification argument.

\begin{proposition}
	\label{prop:ora-weak}
	Fix \(M>0\) and \(0<\rho<1/2\). There exists
	\(C_{M,\rho}>0\) such that an ORA block \(Q_T\) of length $T\geq C_{M,\rho}d\log d$ satisfies
	\[
	\Pp\left\{
	\|Q_T\|_\infty
	>
	d^{-1/2+\rho/2}
	\right\}
	\le
	d^{-M}
	\]
	for all sufficiently large \(d\).
\end{proposition}

\begin{proof}
	For every fixed integer \(p\),
	\cite[Proposition~4.3]{JainEtAl2022} gives, uniformly over
	\(j\in[d]\),
	\[
	\E\sum_{i=1}^d|(Q_Te_j)_i|^{2p}
	\le
	C_pd^{1-p}
	\]
	for \(T\geq C_pd\log d\). Consequently, Markov's inequality yields
	\[
	\Pp\left\{
	\max_{i\in[d]}|(Q_Te_j)_i|^2>d^{-1+\rho}
	\right\}
	\le
	C_pd^{1-p\rho}.
	\]
	Take an integer \(p>(M+2)/\rho\). A union bound over \(j\in[d]\) gives
	\[ \prob{\|Q_T\|^2_\infty >d^{-1 + \rho}} \leq C_p d^{2-p\rho} \leq d^{-M}\]
	for all sufficiently large $d$.
\end{proof}

We next relate ordinary ORA to its symmetric counterpart. A
coupling argument shows that the two walks differ only by signed permutations, which may be placed on either side of the
walk. In particular, they have exactly
the same distributions for the three parameters relevant to this
paper.

Let \(\mathcal H_d\) denote the group of signed permutation matrices, that is,
\[
\mathcal H_d
:=
\left\{
D_\xi P_\pi:
\xi\in\{-1,1\}^d,\ \pi\in S_d
\right\}
\subseteq O(d).
\]

\begin{lemma}\label{lem:ora-sora-coupling}
	There exist couplings of an ORA walk \(Q_T^{\mathrm O}\) and an
	S-ORA walk \(Q_T^{\mathrm S}\), both of length \(T\), such that,
	possibly under different couplings,
	\[
	Q_T^{\mathrm S}
	=
	G_T^{\mathrm L}Q_T^{\mathrm O},
	\qquad 
	Q_T^{\mathrm S}
	=
	Q_T^{\mathrm O}G_T^{\mathrm R},
	\]
	where \(G_T^{\mathrm L}\), \(G_T^{\mathrm R} \in \mathcal H_d\) are signed
	permutation matrices. Consequently, for every $I\subseteq [d]$ and every $X\subseteq \mathbb{S}^{d-1}$, we have
	\[
	\norm{Q_T^{\mathrm S}}_\infty
	\stackrel{\mathrm d}=
	\norm{Q_T^{\mathrm O}}_\infty,
	\qquad
	\kappa_I(Q_T^{\mathrm S})
	\stackrel{\mathrm d}=
	\kappa_I(Q_T^{\mathrm O}),
	\qquad
	\Theta_{Q_T^{\mathrm S}}(X)
	\stackrel{\mathrm d}=
	\Theta_{Q_T^{\mathrm O}}(X).
	\]
\end{lemma}

In particular,
\Cref{prop:ora-weak} holds with the same statement for S-ORA. The proof is given in \Cref{app:ora-symmetry}.

The next lemma shows that an orthogonal product of the form \eqref{eq:U}, with independent S-ORA walks on both sides behaves, for all the
parameters considered in this paper, as one longer S-ORA walk.

\begin{lemma}
	\label{lem:sora-bridge}
	Let $Q_1$ and $Q_2$ be independent S-ORA walks of lengths $T_1$ and $T_2$, respectively, and let
	\(\Sigma\in\mathcal H_d\) be independent of them. Let
	\(F\colon O(d)\to\R\) be invariant under either left or right
	multiplication by signed permutation matrices. Then
	\[
	F(Q_2\Sigma Q_1)
	\stackrel{\mathrm d}=
	F(Q_{T_1+T_2}^{\mathrm S}),
	\]
	where $Q_{T_1+T_2}^{\mathrm S}$ is an S-ORA walk of length $T_1+T_2$.
\end{lemma}

\begin{proof}
	Suppose first that \(F(UG)=F(U)\) for every
	\(G\in\mathcal H_d\). Set
	\[
	\widetilde Q_1
	=
	\Sigma Q_1\Sigma^{-1}.
	\]
	Conjugation by \(\Sigma\) only relabels the selected coordinate
	pair and possibly reverses the angle. Since the S-ORA selected pair is
	uniform and its angle distribution is invariant under
	\(\theta\mapsto-\theta\), the matrix \(\widetilde Q_1\) is a S-ORA walk of length $T_1$
	independent of $Q_2$. Hence
	\[
	Q_2\Sigma Q_1
	=
	Q_2\widetilde Q_1\Sigma,
	\]
	and therefore
	\[
	F(Q_2\Sigma Q_1)
	=
	F(Q_2\widetilde Q_1)
	\stackrel{\mathrm d}=
	F(Q_{T_1+T_2}^{\mathrm S}).
	\]
	If instead \(F(GU)=F(U)\), write
	\[
	Q_2\Sigma Q_1
	=
	\Sigma
	\bigl(\Sigma^{-1}Q_2\Sigma\bigr)Q_1
	\]
	and argue in the same way.
\end{proof}

The absolute maximal entry of a matrix is invariant under signed permutations on either
side. Also, sampled column energy is invariant under right multiplication, and
the parameter \(\Theta_U(X)\) is invariant under left multiplication.
Consequently, for every \(I\subseteq [d]\) and every
\(X\subseteq \mathbb S^{d-1}\),
\[
\norm{Q_2\Sigma Q_1}_\infty
\stackrel{\mathrm d}=
\norm{Q_{T_1+T_2}^{\mathrm S}}_\infty,
\qquad
\kappa_I(Q_2\Sigma Q_1)
\stackrel{\mathrm d}=
\kappa_I(Q_{T_1+T_2}^{\mathrm S}),
\qquad 
\Theta_{Q_2\Sigma Q_1}(X)
\stackrel{\mathrm d}=
\Theta_{Q_{T_1+T_2}^{\mathrm S}}(X).
\]

The preliminary entrywise bounds for two short S-ORA walks can be amplified to obtain substantially stronger control. Combined with the preceding symmetry reductions, this yields Theorem \ref{thm:one-block-ora}.

\begin{proof}[Proof of Theorem \ref{thm:one-block-ora}]
	Write \(T=T_1 + T_2\), where \(T_1,T_2 \geq C_0d\log d\) and
	\(C_0\) is large enough so that
	\Cref{prop:ora-weak} applies to \(T_1\) and $T_2$ with
	\[
	\rho=\frac{1}{4}
	\qquad\text{and}\qquad
	M=4.
	\]
	Let \(Q_1,Q_2\) be independent S-ORA walks of lengths \(T_1\), $T_2$, respectively. By
	\Cref{prop:ora-weak,lem:ora-sora-coupling}, the event
	\[
	E
	=
	\left\{
	\norm{Q_1}_\infty,\norm{Q_2}_\infty
	\le
	d^{-3/8}
	\right\}
	\]
	satisfies
	\[
	\Pp(E^c)
	\le
	2d^{-4}.
	\]
	Let \(D_\xi\) be a Rademacher diagonal, let \(\pi\) be an
	independent uniform permutation, and set
	\[
	U=Q_2D_\xi P_\pi Q_1.
	\]
	On the event \(E\), we have
	\[
	\norm{Q_2}_\infty\norm{Q_1}_\infty
	\le
	d^{-3/4}.
	\]
	Set $\eta_0=d^{-2}/2$. For all large enough $d$ we have $\log(2d^3) \leq \sqrt{d}$, and thus
	\[
	\norm{Q_2}_\infty\norm{Q_1}_\infty
	\le
	\frac{1}{
		\sqrt{d\log(d/\eta_0)}
	}.
	\]
	Consequently, for every realization of \(Q_1,Q_2\) in $E$,
	\Cref{cor:coherence-amplification} applies with failure probability
	\(\eta_0\), and gives
	\[
	\norm U_\infty
	\le
	4\sqrt{\frac{\log(2d^3)}{d}}\leq 8\sqrt{\frac{\log d}{d}},
	\]
	where we have used $d\geq2$. By \Cref{lem:sora-bridge} and \Cref{lem:ora-sora-coupling},
	\[
	\norm U_\infty
	\stackrel{\mathrm d}=
	\norm{Q_T^{\mathrm S}}_\infty
	\stackrel{\mathrm d}=
	\norm{Q_T}_\infty.
	\]
	Therefore,
	\[
	\Pp\left\{
	\norm{Q_T}_\infty
	>
	8\sqrt{\frac{\log d}{d}}
	\right\}
	\le
	\eta_0+\Pp(E^c)
	\le
	\frac{d^{-2}}{2}+2d^{-4}
	\le
	d^{-2},
	\]
	which proves \eqref{eq:6.1}.
\end{proof}

A similar argument shows that an ORA walk of length $O(d \log d)$ has constant sampled column energy throughout the following admissible range.

\begin{proposition}\label{prop:ora-column-energy}
	Fix \(0<\delta<1\). There exists \(C_\delta>0\) such that, for every
	$d$ sufficiently large, the following holds. Let \(Q_T\) be an ORA walk of length $T\geq C_\delta d\log d$. For every fixed \(I\subseteq [d]\) with \(|I|=m\), if
	\[
	m\ge600\log(16d^3)
	\qquad \text{and} \qquad
	36m\log(16d^3m)
	\le
	d^{1-\delta},
	\]
	then
	\[
	\Pp\left\{
	\kappa_I(Q_T)>2
	\right\}
	\le
	d^{-2}.
	\]
\end{proposition}

\begin{proof}
	Write \(T=T_1+T_2\), where \(T_1, T_2\geq C_0d\log d\) and
	\(C_0=C_0(\delta)\) is large enough so that
	\Cref{prop:ora-weak} applies to \(T_1\) and $T_2$ with
	\[
	\rho=\frac{\delta}{2}
	\qquad\text{and}\qquad
	M=4.
	\]
	Let $Q_1$ and $Q_2$ be independent S-ORA walks of lengths \(T_1\) and $T_2$, respectively. By
	\Cref{prop:ora-weak,lem:ora-sora-coupling}, the event
	\[
	E
	=
	\left\{
	\norm{Q_1}_\infty,\norm{Q_2}_\infty
	\le
	d^{-1/2+\delta/4}
	\right\}
	\]
	satisfies
	\[
	\Pp(E^c)
	\le
	2d^{-4}.
	\]
	Let \(D_\xi\) be a Rademacher diagonal, let \(\pi\) be an
	independent uniform permutation, and set
	\[
	U=Q_2D_\xi P_\pi Q_1.
	\]
	On the event \(E\), we have
	\[
	\kappa_I(Q_2)
	\le
	d\norm{Q_2}_\infty^2
	\le
	d^{\delta/2},
	\qquad
	\norm{Q_1}_\infty^2
	\le
	d^{-1+\delta/2}.
	\]
	Hence
	\[
	\kappa_I(Q_2)
	\norm{Q_1}_\infty^2
	\le
	d^{-1+\delta}.
	\]
	For \(\eta_0=d^{-2}/2\), the logarithmic parameters in
	\Cref{cor:leverage-amplification} are
	\[
	\ell
	=
	\log\left(\frac{8d}{\eta_0}\right)
	=
	\log(16d^3),
	\qquad
	L_m
	=
	\log\left(\frac{8dm}{\eta_0}\right)
	=
	\log(16d^3m).
	\]
	Thus, the assumptions are satisfied. For every realization of $Q_1, Q_2$ in $E$, Corollary \ref{cor:leverage-amplification} applies with failure probability
	\(\eta_0\), and yields
	\[
	\kappa_I(U)\le2.
	\]
	Combining the conditional failure probability with
	\(\Pp(E^c)\le2d^{-4}\), we obtain
	\[
	\Pp\left\{
	\kappa_I(U)>2
	\right\}
	\le
	\frac{d^{-2}}{2}+2d^{-4}
	\le
	d^{-2}.
	\]
	Finally, \Cref{lem:sora-bridge} and
	\Cref{lem:ora-sora-coupling} transfer the conclusion to \(Q_T\).
\end{proof}

By the general coherence estimate, Theorem \ref{thm:one-block-ora} also implies that, for
every \(X\subseteq \mathbb S^{d-1}\),
\[
\Theta_{Q_T}(X)
\le
d\norm{Q_T}_\infty^2
\le
C\log d
\]
with high probability. Thus, together with the sampled column energy estimate, Theorem \ref{thm:finite-JL} can be applied to ORA blocks throughout the corresponding admissible range. Moreover, the law of S-ORA is invariant under conjugation by signed permutation matrices. Repeating the argument used in \Cref{cor:kac-JL}, and then using \Cref{lem:sora-bridge,lem:ora-sora-coupling} to pass back to ordinary ORA, yields an analogous result for a single ORA walk. In particular, the conclusion of \Cref{cor:kac-JL} remains valid with the Kac walk replaced by an ordinary ORA walk of length \(O_\delta(d\log d)\), under the corresponding admissibility condition. We omit the formal statement to avoid repetition.

\subsection{Removing the factor \texorpdfstring{\(\log\log d\)}{log log d} from the first stage of ORA}

Krahmer and Ward \cite{KrahmerWard2011} showed that
multiplying a matrix satisfying the RIP at an appropriate sparsity level by a Rademacher diagonal matrix yields a Johnson--Lindenstrauss embedding for finite sets. On the other hand, a result of Dirksen \cite{Dirksen2015} shows that randomly subsampled rows of an orthogonal
matrix with sufficiently small entries satisfy the RIP with high
probability. In \cite{JainEtAl2022}, these two
ingredients are combined in the first stage by running
ORA for
\[
T=O(d\log d\log\log d)
\]
steps until the resulting orthogonal matrix $Q_T$ satisfies
\begin{equation}\label{eq:requiredbound} \|Q_T\|_\infty = O\left (\sqrt{ \log d / d}\right ).
\end{equation}
Then they randomly subsample its rows and use the resulting RIP
matrix to obtain a Johnson--Lindenstrauss embedding into an
intermediate dimension. A second ORA stage then reduces this
intermediate dimension to the optimal order
\(\varepsilon^{-2}\log N\). This approach allows their construction to
reach, up to logarithmic factors, the square-root regime.

The only point in this argument that requires
\(O(d\log d\log\log d)\) updates is the coherence estimate
\eqref{eq:requiredbound} for the first ORA stage.
By \Cref{thm:one-block-ora}, the same estimate holds after only
\(O(d\log d)\) updates. The remaining arguments in the proof of
\cite[Theorem~1.6]{JainEtAl2022} are unchanged.

We also note that the total-variation estimates for S-ORA in
\cite[Lemmas~2.3 and~2.4]{JainEtAl2022} already have the required
accuracy after \(O(d\log d)\) steps. Thus, the corresponding
symmetrizations used in their proof remain valid at the shorter time
scale.

\begin{proof}[Proof of Theorem \ref{thm:jps-improvement}]
	If \(\eps^{-2}\log N>d\), it suffices to take \(m=d\) and use that an ORA walk is an isometry. We may therefore assume \(\eps^{-2}\log N\le d\).
	
	We follow the proof of \cite[Theorem~1.6]{JainEtAl2022}, changing
	only the length of the initial ORA walk \(Q_{T}\). By
	\Cref{thm:one-block-ora}, for \(T\ge C d\log d\), with \(C\)
	sufficiently large,
	\[
	\norm{Q_{T}}_\infty
	\le
	8\sqrt{\frac{\log d}{d}}
	\]
	with probability at least \(1-d^{-2}\). Thus \(Q_{T}\) satisfies
	the same boundedness hypothesis used in
	\cite[Theorem~3.4]{JainEtAl2022}; in the original proof this is
	obtained from \cite[Proposition~4.4]{JainEtAl2022} after
	\(O(d\log d\log\log d)\) updates.
	
	Hence the remainder of the proof of
	\cite[Theorem~1.6]{JainEtAl2022}, including the choice of the
	intermediate dimension and the second ORA stage, is unchanged after
	adjusting universal constants. The resulting map has \(m\le C\eps^{-2}\log N\) and
	succeeds with probability at least \(2/3\). The only change in the
	running time is that the initial ORA walk now costs
	\(O(d\log d)\). All remaining terms and the \(O(1)\) additional
	working memory bound are unchanged.
\end{proof}

\subsection{Parallel orthogonal repeated averaging}
\label{subsec:parallel-ora}

We begin by recalling the parallel version of the Kac walk introduced in \cite{LuQinSongYaoZhao2024}. Assume that \(d\) is even. In one round,
the walk chooses a uniform perfect matching
\[
\mathcal M
=
\bigl\{
(i_1,j_1),\ldots,(i_{d/2},j_{d/2})
\bigr\}
\]
of \([d]\), samples independent angles $\theta_1,\ldots,\theta_{d/2}$ uniformly from \([0,2\pi)\), and simultaneously applies the rotations
\[
R_{\theta_\ell}
=
\begin{pmatrix}
	\cos\theta_\ell & -\sin\theta_\ell\\
	\sin\theta_\ell & \cos\theta_\ell
\end{pmatrix}
\]
to the coordinate pairs \((i_\ell,j_\ell)\).

We now prove \Cref{thm:parallel-ora-coherence}. Recall that one parallel ORA round chooses a uniform permutation \(\sigma\in S_d\), organizes the
coordinates into the ordered pairs
\[
(\sigma(1),\sigma(2)),\ldots,
(\sigma(d-1),\sigma(d)),
\]
and applies the rotation
\[
R_{\pi/4}
=
\frac1{\sqrt2}
\begin{pmatrix}
	1&-1\\
	1&1
\end{pmatrix}
\]
simultaneously to every pair. We denote the resulting orthogonal
matrix by \(R_\sigma^{\mathrm{par}}\), and write
\[
Q_T^{\mathrm{par}}
:=
R_{\sigma_T}^{\mathrm{par}}
\cdots
R_{\sigma_1}^{\mathrm{par}}
\]
for the product of \(T\) independent rounds.  We show that parallel ORA reaches the same near optimal coherence as
ordinary ORA after \(O(\log d)\) parallel rounds. Thus, for the purpose of obtaining near optimal coherence, the independent random rotations in a parallel Kac round can be replaced by the fixed angle $\pi/4$. We do not address the stronger mixing or pseudorandomness properties considered in \cite{LuQinSongYaoZhao2024,LuQinSongYaoZhao2025}.

The proof follows the same two-step strategy as for ordinary ORA.
Fixed-order moments first provide a preliminary bound on the
largest entries of two independent walks. Next, the coherence amplification
result upgrades this bound to near optimal coherence. Finally, the signed permutation symmetries of parallel ORA identify this auxiliary product with a longer parallel ORA walk.

\begin{lemma}
	\label{lem:parallel-ora-moments}
	Let $d$ be even. For every fixed integer \(p\ge1\), there exists a constant \(C_p>0\) such that for any $x \in \mathbb{S}^{d-1}$ we have
	\[
	\E
	\sum_{i=1}^d
	\left|
	(Q_T^{\mathrm{par}}x)_i
	\right|^{2p}
	\le
	C_p d^{1-p} \qquad \text{for } T\geq C_p \log d.
	\]
\end{lemma}

The proof of \Cref{lem:parallel-ora-moments} is analogous to the moment
argument in \cite[Proposition~4.3]{JainEtAl2022} and is given in
\Cref{app:parallel-ora-moments}.

\begin{proof}[Proof of Theorem \ref{thm:parallel-ora-coherence}]
	Let $p \in \mathbb{N}$ and let $Q_1$ be a parallel ORA walk of length \(T_1\geq C_p\log d\). By
	\Cref{lem:parallel-ora-moments},
	\[
	\E
	\sum_{i=1}^d
	\left|
	(Q_1e_j)_i
	\right|^{2p}
	\le
	C_p d^{1-p}
	\]
	uniformly over \(j\in[d]\). Consequently, Markov's inequality and a union bound over the columns give
	\[
		\Pp\left\{
		\norm{Q_1}_\infty>d^{-3/8}
		\right\}
		\le
		C_p d^{2-p/4}.
	\]
	Taking \(p=28\), for $d\ge C_p$ we obtain
	\begin{equation}\label{eq:parallel-ora-weak-delocalization}
		\Pp\left\{
		\norm{Q_1}_\infty>d^{-3/8}
		\right\}
		\le d^{-4}.
	\end{equation}
	
	Let \(Q_2\) be an independent parallel ORA walk of length
	\(T_2\geq C_p \log d\), let \(D_\xi\) be an independent Rademacher diagonal matrix, and let \(P_\pi\) be an independent uniform permutation matrix. Set
	\[
	U=Q_2D_\xi P_\pi Q_1.
	\]
	Define the event
	\[
	E
	=
	\left\{
	\norm{Q_1}_\infty,\norm{Q_2}_\infty
	\le
	d^{-3/8}
	\right\}.
	\]
	By \eqref{eq:parallel-ora-weak-delocalization} and a union bound, $\Pp(E^c)\le 2d^{-4}$. On the event \(E\),
	\[
	\norm{Q_1}_\infty\norm{Q_2}_\infty
	\le
	d^{-3/4}.
	\]
	For all sufficiently large \(d\), this satisfies the hypothesis of
	\Cref{cor:coherence-amplification} with failure probability
	\(\eta=d^{-2}/2\). Hence, for every realization of $Q_1,Q_2$ in $E$, 
	\[
	\Pp_{\xi,\pi}
	\left\{
	\norm{U}_\infty
	>
	C\sqrt{\frac{\log d}{d}}
	\right\}
	\le
	\frac{d^{-2}}{2},
	\]
	for some absolute constant \(C>0\). Therefore, using $d\geq 2$,
	\begin{equation}\label{eq:parallel-ora-amplified}
		\Pp
		\left\{
		\norm{U}_\infty
		>
		C\sqrt{\frac{\log d}{d}}
		\right\}
		\le
		\frac{d^{-2}}{2}+2d^{-4}
		\le
		d^{-2}.
	\end{equation}
	
	It remains to identify the distribution of \(\norm{U}_\infty\).
	We claim that, for one parallel ORA round \(R\) and every signed
	permutation matrix \(\Sigma\), there exist a signed permutation
	matrix \(\Sigma'\) and a parallel ORA round \(\widetilde R\),
	distributed as \(R\), such that
	\[
	R\Sigma=\Sigma'\widetilde R.
	\]
	Indeed, the permutation part of \(\Sigma\) merely relabels the uniform random ordered pairing. For the diagonal-sign part, on each
	matched pair one has
	\[
	R_{\pi/4}
	\begin{pmatrix}
		\xi_1&0\\
		0&\xi_2
	\end{pmatrix}
	=
	\Sigma_{\xi_1,\xi_2}R_{\pi/4},
	\]
	where
	\[
	\Sigma_{\xi_1,\xi_2}
	=
	R_{\pi/4}
	\begin{pmatrix}
		\xi_1&0\\
		0&\xi_2
	\end{pmatrix}
	R_{-\pi/4}
	\]
	is a signed permutation matrix.
	Iterating this identity through the rounds of \(Q_2\), we obtain
	\[
	Q_2D_\xi P_\pi
	=
	\Sigma'\widetilde Q_2,
	\]
	where \(\widetilde Q_2\) has the law of a
	parallel ORA walk of length $T_2$ and is independent of \(Q_1\). Indeed, conditionally on the preceding rounds, each transformed ordered
	pairing remains uniform. Since left multiplication by a signed permutation
	does not change the maximum absolute entry,
	\[
	\norm{U}_\infty
	=
	\norm{\widetilde Q_2Q_1}_\infty
	\stackrel{\mathrm d}=
	\norm{Q^{\mathrm{par}}_{T}}_\infty,
	\]
	where $Q^{\mathrm{par}}_{T}$ is a parallel ORA walk of length $T=T_1+T_2$. The conclusion follows from
	\eqref{eq:parallel-ora-amplified}, after adjusting the universal constant \(C\).
	\end{proof}

\section{A limitation of the local parameter framework}
\label{sec:limitations}

Theorem \ref{thm:finite-JL} is formulated entirely in terms of the sampled
column energy \(\kappa_I(A)\) and the maximal variance
\(\Theta_B(X)\). The purpose of this section is to show that these two
parameters alone, even when they take their optimal values, do not
guarantee a sub-Gaussian tail when \(m\) is of order \(\sqrt d\).

\begin{theorem}
	\label{thm:sqrt-obstruction}
	For $r$ large enough, set $d=2^{2r}$, $m=\sqrt{d}$, and $I=[m]$. There are orthogonal matrices \(A,B\in O(d)\), and a vector \(x\in\mathbb S^{d-1}\) such that
	\[
	\kappa_I(A)=1,
	\qquad
	\Theta_B(\{x\})=1,
	\]
	but for $\Phi\colon \mathbb{R}^d \longrightarrow \mathbb{R}^m$ defined by (\ref{eq:Phi-main}) one has
	\begin{equation}\label{eq:sqrt-obstruction}
		\Pp
		\left\{
		\left|
		\norm{\Phi x}_2^2-1
		\right|
		\ge3
		\right\}
		\ge
		\exp\left[
		-Cd^{1/3}(\log d)^{2/3}
		\right],
	\end{equation}
	where $C>0$ is a universal constant. 
\end{theorem}

The theorem does not rule out the possibility that a particular
transform of the form \(\Phi\) may satisfy a sub-Gaussian tail bound
when \(m\asymp\sqrt d\). Rather, it shows that such a conclusion
cannot, in general, be deduced from the parameters
\(\kappa_I(A)\) and \(\Theta_B(X)\) alone. Any stronger result for a particular orthogonal construction must use further structural information.

\begin{proof}
	Let \(d=2^{2r}\) and $m=\sqrt{d}$. First, we construct the matrix $A$. Partition \([d]\) into \(m\) consecutive blocks $\mathcal B_1,\ldots,\mathcal B_m$, each of size $m$. Let \(J_m\) denote the \(m\times m\) all-ones matrix and define
	\[
	P_0
	=
	\operatorname{diag}
	\left(
	\frac1mJ_m,\ldots,\frac1mJ_m
	\right) \in \mathbb{R}^{d\times d}.
	\]
	Since
	\(m^{-1}J_m\) is the orthogonal projection onto the constant
	vectors in \(\R^m\), the matrix \(P_0\) is an orthogonal
	projection of rank \(m\). Choose \(A\in O(d)\) so that its first $m$ rows form an orthonormal basis of \(\operatorname{range}(P_0)\). Then
	\[
	A^{\mathsf T}P_I^{\mathsf T}P_IA=P_0.
	\]
	Every diagonal entry of \(P_0\) equals $1/m$ and therefore
	\[
	\kappa_I(A)
	=
	\frac dm\max_{j\in[d]}(P_0)_{jj}
	=
	\frac{d}{m^2}
	=
	1.
	\]
	For the right factor, we take $B=H_d$ the normalized Walsh--Hadamard matrix. In particular, $H_d^2=I_d$. Since $\|B\|_\infty=d^{-1/2}$, for every
\(X\subseteq \mathbb S^{d-1}\) one has
\[
\Theta_B(X)
=
d\sup_{\substack{x\in X\\i\in[d]}}
\sum_{j=1}^d B_{ij}^2x_j^2
=
1.
\]
Let \(k\) be a power of two satisfying $4\leq k\leq m/4$ and take $s$ so that
\[
s k^2 = d.
\]
Consider the vector $z \in \mathbb{S}^{d-1}$ whose coordinates are given by
\[
z_j
=
\begin{cases}
	k^{-1}
	&
	\text{if } j\leq k^2
	\\[2mm]
	0,
	&
	\text{otherwise}.
\end{cases}
\]
Write $x=Bz$. A direct calculation from the recursive definition of the Walsh--Hadamard matrix gives
\[
x_j
=
\begin{cases}
	s^{-1/2},
	&
	j\equiv1\pmod{k^2},
	\\[2mm]
	0,
	&
	\text{otherwise}.
\end{cases}
\]
Then \(x\in\mathbb S^{d-1}\) and it has exactly \(s\) nonzero
coordinates. Consider the event
\[
E
=
\left\{
\xi'_j=1
\text{ for every }j\in\supp(x)
\right\}.
\]
Since the coordinates of \(\xi'\) are independent, $\Pp(E)=2^{-s}$. On \(E\), one has \(D_{\xi'}x=x\), and therefore
\[
BD_{\xi'}x=Bx=BBz=z.
\]
Set \(S=[2k] \subseteq [d]\) and 
\[
y=P_\pi z.
\]
After applying the permutation, the coordinates of $S$ move to new locations $T$, where $T$ is a uniform subset of $[d]$ with $2k$ elements. Hence,
\[
\Pp_\pi\{T\subseteq \mathcal B_1\}
=
\prod_{j=0}^{2k-1}
\frac{m-j}{m^2-j}.
\]
Assuming \(k\le m/4\), every factor in the product is at least
\(1/(2m)\), and consequently
\begin{equation}\label{eq:distinguished-packet-event}
	\Pp_\pi\{T\subseteq \mathcal B_1\}
	\ge
	(2m)^{-2k}.
\end{equation}
Condition on \(E\) and on a realization of the permutation for which $T\subseteq \mathcal B_1$.
Consider the event
\[
F
=
\left\{
\xi_j=1
\text{ for every }j\in T
\right\}.
\]
Since the coordinates of \(\xi\) are independent, $\Pp(F)=2^{-2k}$. Observe that
\[ \norm{\Phi x}_2^2 = \frac{d}{m} \|P_I A D_\xi y\|_2^2 = \frac{d}{m} y^\mathsf{T} D_\xi P_0 D_\xi y = \sum_{i=1}^m
\left(
\sum_{j\in\mathcal B_i}\xi_j y_j
\right)^2 \geq \left(\sum_{j\in\mathcal B_1}\xi_j y_j\right)^2. \]
Conditioning on the event $F$, the contribution of the first block can be written as
\[ \sum_{j\in\mathcal B_1}\xi_j y_j
=
\sum_{j\in T}\frac1k+W
=
2+W,\]
where the random variable \(W = \sum_{j\in\mathcal B_1\setminus T}\xi_jy_j\) has a distribution symmetric about
zero (conditionally on \(E\), on the permutation, and on the signs indexed by \(T\)). Hence
\[
\Pp_\xi\{W\ge0\}\ge\frac12.
\]
It follows that, with conditional probability at least
\(2^{-2k-1}\),
\[
\sum_{i\in\mathcal B_1}\xi_i y_i
\ge2,
\]
which implies $\norm{\Phi x}_2^2\ge4$, and thus
\[
\left|
\norm{\Phi x}_2^2-1
\right|
\ge3.
\]

Combining the probabilities of \(E\), the permutation event in
\eqref{eq:distinguished-packet-event}, and the preceding outer-sign
events, we obtain
\begin{equation}\label{eq:7.1}
	\Pp
	\left\{
	\left|
	\norm{\Phi x}_2^2-1
	\right|
	\ge3
	\right\}
	\ge
	2^{-s}(2m)^{-2k}2^{-2k-1}
	\ge
	\exp\left[
	-C\left(
	s+k\log m
	\right)
	\right].
\end{equation}
It remains to optimize the choice of \(k\). For $r$ large enough we can take
\[
k
\asymp
\frac{m^{2/3}}{(\log m)^{1/3}}.
\]
Consequently,
\[
s
+
k\log m
\asymp
m^{2/3}(\log m)^{2/3}.
\]
Substituting this into (\ref{eq:7.1}) and using $m^2=d$ yields
\[
\Pp
\left\{
\left|
\norm{\Phi x}_2^2-1
\right|
\ge3
\right\}
\ge
\exp\left[
-Cd^{1/3}(\log d)^{2/3}
\right].\qedhere
\]
\end{proof}

\appendix

\section{A moment generating function bound for permuted sums}
\label{app:scalar-permutation-mgf}

We recall the permanent inequality of Carlen, Lieb, and Loss
\cite{CarlenLiebLoss2006}: for every complex \(d\times d\) matrix
\(Z\),
\begin{equation}\label{eq:cll-permanent}
	|\operatorname{per}(Z)|
	\le
	d!
	\prod_{i=1}^d
	\left(
	\frac{1}{d}\sum_{j=1}^d |Z_{ij}|^2
	\right)^{1/2}.
\end{equation}

\begin{proof}[Proof of \Cref{lem:scalar-comb-bern}]
	For \(\lambda\in\R\), define the matrix $Z \in \mathbb{R}^{d\times d}$ entrywise by
	\[
	Z_{ij}
	=
	\exp(\lambda a_i b_j).
	\]
	By the definition of the permanent \eqref{eq:per},
	\[
	\E_\pi
	\exp\left(
	\lambda\sum_{i=1}^d a_i b_{\pi(i)}
	\right)
	= \frac{1}{d!} \sum_{\pi \in S_d} \exp \left (\sum_{i=1}^d \lambda a_i b_{\pi(i)} \right ) = \frac{1}{d!} \sum_{\pi \in S_d} \prod_{i=1}^d Z_{i,\pi(i)}= 
	\frac{1}{d!}\operatorname{per}(Z).
	\]
	Applying \eqref{eq:cll-permanent}, we obtain
	\[
	\E_\pi
	\exp\left(
	\lambda\sum_{i=1}^d a_i b_{\pi(i)}
	\right)
	\le
	\prod_{i=1}^d
	\left(
	\frac1d
	\sum_{j=1}^d
	\exp(2\lambda a_i b_j)
	\right)^{1/2}.
	\]
	Note that for $|\lambda|\le 1/M$ we have $|2\lambda a_i b_j|\leq 2$. Using the elementary inequality
	\[
	e^u
	\le
	1+u+2u^2,
	\qquad
	\text {for }|u|\le2,
	\]
	and the centering condition \(\sum_j b_j=0\), we get
	\[
	\begin{aligned}
		\frac1d
		\sum_{j=1}^d
		\exp(2\lambda a_i b_j)
		&\le
		1
		+
		\frac{2\lambda a_i}{d}
		\sum_{j=1}^d b_j
		+
		\frac{8\lambda^2a_i^2}{d}
		\sum_{j=1}^d b_j^2
		=
		1
		+
		8\lambda^2a_i^2
		\frac{\norm{b}_2^2}{d}.
	\end{aligned}
	\]
	Therefore, using \(\log(1+u)\le u\),
	\[
		\log
		\E_\pi
		\exp\left(
		\lambda\sum_{i=1}^d a_i b_{\pi(i)}
		\right)
		\le
		\frac12
		\sum_{i=1}^d
		\log\left(
		1+
		8\lambda^2a_i^2
		\frac{\norm{b}_2^2}{d}
		\right)
		\le
		4\lambda^2
		\frac{\norm{a}_2^2\norm{b}_2^2}{d}.\qedhere
	\]
\end{proof}

\section{Coupling ORA and S-ORA}
\label{app:ora-symmetry}

\begin{proof}[Proof of \Cref{lem:ora-sora-coupling}]
	Recall that \(R_{ij,\theta}\) denotes the rotation by angle \(\theta\)
	in the coordinate plane spanned by \(e_i\) and \(e_j\), acting as the
	identity on the remaining coordinates. We write
	\(R_{ij}=R_{ij,\pi/4}\). We first observe that, for every signed permutation
	\(G\in\mathcal H_d\) and every ordered pair \(i\neq j\), there is
	an ordered pair \(i'\neq j'\) such that
	\[
	G^{-1}R_{ij}G=R_{i'j'}.
	\]
	Indeed, conjugation by a signed permutation relabels the coordinate
	plane and may reverse its orientation. In the latter case we use
	\(R_{ij,-\pi/4}=R_{ji,\pi/4}\). For fixed \(G\), the resulting map
	\((i,j)\mapsto(i',j')\) is a bijection of the ordered pairs of
	distinct coordinates.
	
	An S-ORA update on a pair \((i,j)\) differs from the ordinary ORA
	update on the same pair only by a signed permutation of the two
	coordinates. Indeed, if
	\[
	\theta\in
	\left\{
	\frac\pi4,\frac{3\pi}4,\frac{5\pi}4,\frac{7\pi}4
	\right\},
	\]
	then \(R_{ij,\theta}R_{ij,-\pi/4}\) is a rotation by a multiple of
	\(\pi/2\), and hence is a signed permutation matrix.
	We now construct the coupling inductively. Set
	\(Q_0^{\mathrm S}=Q_0^{\mathrm O}=G_0=I_d\), and suppose that after
	\(t-1\) steps
	\[
	Q_{t-1}^{\mathrm S}=G_{t-1}Q_{t-1}^{\mathrm O}
	\]
	for some \(G_{t-1}\in\mathcal H_d\). Let the next S-ORA update act on
	the ordered pair \((i_t,j_t)\). As observed above, it can be written as
	\(H_tR_{i_tj_t}\) for some \(H_t\in\mathcal H_d\).
	
	By the conjugation property, there is an ordered pair
	\((i_t',j_t')\) such that
	\[
	R_{i_tj_t}G_{t-1}
	=
	G_{t-1}R_{i_t'j_t'}.
	\]
	Therefore
	\[
	Q_t^{\mathrm S}
	=
	H_tG_{t-1}R_{i_t'j_t'}Q_{t-1}^{\mathrm O}.
	\]
	Set
	\[
	G_t=H_tG_{t-1},
	\qquad
	Q_t^{\mathrm O}
	=
	R_{i_t'j_t'}Q_{t-1}^{\mathrm O}.
	\]
	For fixed \(G_{t-1}\), the map
	\((i_t,j_t)\mapsto(i_t',j_t')\) is a bijection of the ordered pairs of
	distinct coordinates. Hence \((i_t',j_t')\) is again uniform and
	independent of the past, and thus \(Q_t^{\mathrm O}\) has the law of an ORA walk and
	\[
	Q_t^{\mathrm S}=G_tQ_t^{\mathrm O}.
	\]
	Iterating the construction produces the desired coupling.
	
	The right factorization follows by taking inverses. Indeed,
	\(R_{ij}^{-1}=R_{ji}\), so an ORA walk has the same
	distribution as its inverse. The same is true for S-ORA, since its
	angle distribution is invariant under \(\theta\mapsto-\theta\).
	Thus, taking inverses in the preceding coupling gives, possibly
	under a different coupling,
	\[
	Q_T^{\mathrm S}
	=
	Q_T^{\mathrm O}G_T^{\mathrm R}
	\]
	for some \(G_T^{\mathrm R}\in\mathcal H_d\).
	
	Finally, left multiplication by a signed permutation leaves
	\(\norm{\cdot}_\infty\) and \(\Theta_{(\cdot)}(X)\) unchanged, while
	right multiplication leaves	\(\kappa_I(\cdot)\) unchanged.
\end{proof}

\section{Moments of parallel ORA}
\label{app:parallel-ora-moments}

\begin{proof}[Proof of \Cref{lem:parallel-ora-moments}]
	The case \(p=1\) follows immediately from orthogonality. Fix
	\(p\ge2\), and for \(1\le r\le p\) set
	\[
	S_r(x)
	:=
	\frac1{(2r)!}\sum_{i=1}^d x_i^{2r}.
	\]
	Let \(R\) denote one parallel ORA round and suppose that \(i\) and \(j\)
	are paired. The contribution of these coordinates to $S_p(Rx)$ is
	\[
		\frac1{(2p)!}
		\left (
		\left(\frac{x_i+x_j}{\sqrt2}\right)^{2p}
		+
		\left(\frac{x_i-x_j}{\sqrt2}\right)^{2p}
		\right )
		=
		\frac{2^{-p}}{(2p)!}
		\left(
		(x_i+x_j)^{2p}+(x_i-x_j)^{2p}
		\right ).
	\]
	The binomial theorem gives
	\[ (x_i + x_j)^{2p} = \sum_{r=0}^{2p} {2p \choose r} x_i^r x_j^{2p-r}
	\qquad
	\text{and}
	\qquad
	 (x_i - x_j)^{2p} = \sum_{r=0}^{2p} (-1)^r {2p \choose r} x_i^r x_j^{2p-r}\]
	When adding these expressions odd terms cancel and we get
	\[ (x_i + x_j)^{2p} + (x_i - x_j)^{2p} = 2\sum_{r=0}^p {2p \choose 2r} x_i^{2r} x_j^{2p-2r}.\]
	For a uniform perfect matching, each unordered pair \(\{i,j\}\)
	occurs with probability \(1/(d-1)\). Therefore
	\[
		\E_R S_p(Rx)
		=
		\frac{2^{1-p}}{d-1}
		\sum_{1\le i<j\le d}
		\sum_{a=0}^p
		\frac{x_i^{2a}x_j^{2p-2a}}
		{(2a)!(2p-2a)!}
		=
		\frac{2^{-p}}{d-1}
		\sum_{r=0}^p
		\sum_{i\ne j}
		\frac{x_i^{2r}x_j^{2p-2r}}
		{(2r)!(2p-2r)!}.
	\]
	The terms \(r=0\) and \(r=p\) together contribute exactly
	\(2^{1-p}S_p(x)\). For \(1\le r\le p-1\), all summands are
	nonnegative, so we have
	\[ \sum_{i\neq j} \frac{x_i^{2r}x_j^{2p-2r}}
	{(2r)!(2p-2r)!} \leq \sum_{i,j=1}^d \frac{x_i^{2r}x_j^{2p-2r}}
	{(2r)!(2p-2r)!} = S_r(x)S_{p-r}(x).\]
	Thus, it follows that
	\begin{equation}\label{eq:appendixC1}
	\E_R S_p(Rx)
	\le
	2^{1-p}S_p(x)
	+
	\frac{2^{-p}}{d-1}
	\sum_{r=1}^{p-1}
	S_r(x)S_{p-r}(x).
	\end{equation}
	We next control the lower moments in terms of \(S_p(x)\) via interpolation. Fix $1\leq r \leq p$ and set
	\[ \alpha= \frac{p-r}{p-1}, \qquad \beta=\frac{r-1}{p-1}.\]
	Observe that $\alpha+\beta=1$ and $2\alpha + 2p\beta = 2r$. H\"older's inequality then yields
	\begin{align*}
		\sum_{i=1}^d |x_i|^{2r}
		&= \sum_{i=1}^d \left (|x_i|^2\right )^\alpha \left (|x_i|^{2p}\right )^\beta\le
		\left(
		\sum_{i=1}^d |x_i|^2
		\right)^{\alpha}
		\left(
		\sum_{i=1}^d |x_i|^{2p}
		\right)^{\beta}
		\\
		&=
		\left(
		\sum_{i=1}^d |x_i|^{2p}
		\right)^{(r-1)/(p-1)},
	\end{align*}
	where we have used that $x \in \mathbb{S}^{d-1}$. Hence, for \(1\le r\le p-1\),
	\[
	S_r(x)S_{p-r}(x)
	\le
	C_p
	S_p(x)^{(p-2)/(p-1)},
	\]
	where \(C_p\) depends only on \(p\). Since \(d\ge2\), from \eqref{eq:appendixC1} and the above estimate, after adjusting \(C_p\), we obtain
	\[
	\E_R S_p(Rx)
	\le
	2^{1-p}S_p(x)
	+
	\frac{C_p}{d}
	S_p(x)^{(p-2)/(p-1)}.
	\]
	Set \(u=S_p(x)\). We distinguish whether \(u\) is above or below
	the target scale \(d^{1-p}\). For \(K_p>0\), define
	\[
	\rho_p
	:=
	2^{1-p}+C_pK_p^{-1/(p-1)}.
	\]
	Choose \(K_p\) sufficiently large so that \(\rho_p<1\). If \(u\ge K_pd^{1-p}\), then
	\[
	\frac1d u^{(p-2)/(p-1)}
	=
	\frac{u}{d\,u^{1/(p-1)}}
	\le
	K_p^{-1/(p-1)}u,
	\]
	and hence
	\[
	\E_R S_p(Rx)
	\le
	\rho_p u.
	\]
	If instead \(u<K_pd^{1-p}\), then
	\[
	\E_R S_p(Rx)
	\le
	2^{1-p}K_pd^{1-p}
	+
	C_pK_p^{(p-2)/(p-1)}d^{1-p}
	\le
	C_p' d^{1-p}.
	\]
	Combining the two cases and adjusting the constant gives
	\[
	\E_R S_p(Rx)
	\le
	\rho_p S_p(x)+C_p d^{1-p}.
	\]
	Now denote
	\[
	x_t=Q_t^{\mathrm{par}}x.
	\]
	The \((t+1)\)-st parallel ORA round is independent of \(x_t\).
	Thus, conditioning on \(x_t\),
	\[
	\E\left[
	S_p(x_{t+1})
	\,\middle|\,
	x_t
	\right]
	\le
	\rho_p S_p(x_t)+C_p d^{1-p}.
	\]
	Taking expectations and iterating gives
	\[
		\E S_p(x_t)
		\le
		\rho_p^tS_p(x)
		+
		C_p d^{1-p}
		\sum_{s=0}^{t-1}\rho_p^s
		\le
		\rho_p^tS_p(x)+C_p d^{1-p},
	\]
	after adjusting \(C_p\). Finally, since \(x\in\mathbb S^{d-1}\),
	\[
	S_p(x)
	=
	\frac1{(2p)!}
	\sum_{i=1}^d |x_i|^{2p}
	\le
	\frac1{(2p)!}.
	\]
	Since \(\rho_p<1\), after increasing \(C_p\), taking
	\(T\ge C_p\log d\) implies
	\[
	\rho_p^TS_p(x)
	\le
	d^{1-p}.
	\]
	Consequently,
	\[
	\E S_p(Q_T^{\mathrm{par}}x)
	\le
	C_p d^{1-p}.
	\]
	Multiplying by \((2p)!\) and absorbing this factor into \(C_p\) proves the result.
\end{proof}

\bibliographystyle{amsplain}
\bibliography{references}

@misc{ChengBarber2026,
  author       = {Cheng, Chen and Barber, Rina Foygel},
  title        = {Concentration Inequalities for Exchangeable Tensors and Matrix-Valued Data},
  howpublished = {arXiv:2601.20152},
  year         = {2026},
  note         = {Preprint}
}

@article{AilonChazelle2009,
  author  = {Ailon, Nir and Chazelle, Bernard},
  title   = {The Fast Johnson--Lindenstrauss Transform and Approximate Nearest Neighbors},
  journal = {SIAM Journal on Computing},
  volume  = {39},
  number  = {1},
  pages   = {302--322},
  year    = {2009},
  doi     = {10.1137/060673096}
}

@article{AilonLiberty2009,
  author  = {Ailon, Nir and Liberty, Edo},
  title   = {Fast Dimension Reduction Using Rademacher Series on Dual {BCH} Codes},
  journal = {Discrete \& Computational Geometry},
  volume  = {42},
  number  = {4},
  pages   = {615--630},
  year    = {2009},
  doi     = {10.1007/s00454-009-9167-x}
}

@incollection{Albert2019,
  author    = {Albert, M{\'e}lisande},
  title     = {Concentration Inequalities for Randomly Permuted Sums},
  booktitle = {High Dimensional Probability VIII},
  series    = {Progress in Probability},
  volume    = {74},
  pages     = {341--383},
  year      = {2019},
  publisher = {Birkh{\"a}user},
}

@article{CarlenLiebLoss2006,
  author  = {Carlen, Eric A. and Lieb, Elliott H. and Loss, Michael},
  title   = {An Inequality of Hadamard Type for Permanents},
  journal = {Methods and Applications of Analysis},
  volume  = {13},
  number  = {1},
  pages   = {1--18},
  year    = {2006}
}

@article{Dirksen2015,
  author  = {Dirksen, Sjoerd},
  title   = {Tail Bounds via Generic Chaining},
  journal = {Electronic Journal of Probability},
  volume  = {20},
  pages   = {1--29},
  year    = {2015},
  doi     = {10.1214/EJP.v20-3888}
}

@article{JainEtAl2022,
  author  = {Jain, Vishesh and Pillai, Natesh S. and Sah, Ashwin and Sawhney, Mehtaab and Smith, Aaron},
  title   = {Fast and Memory-Optimal Dimension Reduction Using Kac's Walk},
  journal = {The Annals of Applied Probability},
  volume  = {32},
  number  = {5},
  pages   = {4038--4064},
  year    = {2022},
  doi     = {10.1214/22-AAP1784}
}

@misc{JainMizgerd2026,
  author       = {Jain, Vishesh and Mizgerd, Clayton},
  title        = {Total Variation Cutoff for Kac's Walk on the Sphere},
  howpublished = {arXiv:2607.13401},
  year         = {2026},
  note         = {Preprint}
}

@incollection{JohnsonLindenstrauss1984,
  author    = {Johnson, William B. and Lindenstrauss, Joram},
  title     = {Extensions of Lipschitz Mappings into a Hilbert Space},
  booktitle = {Conference in Modern Analysis and Probability},
  series    = {Contemporary Mathematics},
  volume    = {26},
  pages     = {189--206},
  publisher = {American Mathematical Society},
  year      = {1984}
}

@article{KrahmerWard2011,
  author  = {Krahmer, Felix and Ward, Rachel},
  title   = {New and Improved Johnson--Lindenstrauss Embeddings via the Restricted Isometry Property},
  journal = {SIAM Journal on Mathematical Analysis},
  volume  = {43},
  number  = {3},
  pages   = {1269--1281},
  year    = {2011},
  doi     = {10.1137/100810447}
}

@article{MackeyEtAl2014,
  author  = {Mackey, Lester and Jordan, Michael I. and Chen, Richard Y. and Farrell, Brendan and Tropp, Joel A.},
  title   = {Matrix Concentration Inequalities via the Method of Exchangeable Pairs},
  journal = {The Annals of Probability},
  volume  = {42},
  number  = {3},
  pages   = {906--945},
  year    = {2014},
  doi     = {10.1214/13-AOP892}
}

@article{Barber2024,
  author  = {Barber, Rina Foygel},
  title   = {Hoeffding and Bernstein Inequalities for Weighted Sums of
             Exchangeable Random Variables},
  journal = {Electronic Communications in Probability},
  volume  = {29},
  pages   = {1--13},
  year    = {2024},
  doi     = {10.1214/24-ECP620}
}

@book{Vershynin2026,
  author    = {Vershynin, Roman},
  title     = {High-Dimensional Probability:
               An Introduction with Applications in Data Science},
  edition   = {Second},
  publisher = {Cambridge University Press},
  year      = {2026}
}

@article{BrudnyiYomdin2015,
  author  = {Brudnyi, Alexander and Yomdin, Yosef},
  title   = {Norming Sets and Related {R}emez-Type Inequalities},
  journal = {Journal of the Australian Mathematical Society},
  volume  = {100},
  number  = {2},
  pages   = {163--181},
  year    = {2016},
  doi     = {10.1017/S1446788715000488}
}

@misc{ZhangKileel2025,
  author       = {Zhang, Yifan and Kileel, Joe},
  title        = {Norming Sets for Tensor and Polynomial Sketching},
  howpublished = {arXiv:2506.05174},
  year         = {2025},
  note         = {Preprint}
}

@article{BaraniukEtAl2008,
  author  = {Baraniuk, Richard G. and Davenport, Mark A. and
             DeVore, Ronald A. and Wakin, Michael B.},
  title   = {A Simple Proof of the Restricted Isometry Property for
             Random Matrices},
  journal = {Constructive Approximation},
  volume  = {28},
  number  = {3},
  pages   = {253--263},
  year    = {2008},
  doi     = {10.1007/s00365-007-9003-x}
}

@article{RechtFazelParrilo2010,
  author  = {Recht, Benjamin and Fazel, Maryam and Parrilo, Pablo A.},
  title   = {Guaranteed Minimum-Rank Solutions of Linear Matrix
             Equations via Nuclear Norm Minimization},
  journal = {SIAM Review},
  volume  = {52},
  number  = {3},
  pages   = {471--501},
  year    = {2010},
  doi     = {10.1137/070697835}
}

@inproceedings{LuQinSongYaoZhao2024,
  author    = {Lu, Chuhan and Qin, Minglong and Song, Fang and
               Yao, Penghui and Zhao, Mingnan},
  title     = {Quantum Pseudorandom Scramblers},
  booktitle = {Theory of Cryptography: 22nd International Conference,
               {TCC} 2024, Milan, Italy, December 2--6, 2024,
               Proceedings, Part {II}},
  editor    = {Boyle, Elette and Mahmoody, Mohammad},
  series    = {Lecture Notes in Computer Science},
  volume    = {15365},
  pages     = {3--35},
  publisher = {Springer Nature Switzerland},
  address   = {Cham},
  year      = {2024},
  doi       = {10.1007/978-3-031-78017-2_1}
}

@misc{LuQinSongYaoZhao2025,
  author       = {Lu, Chuhan and Qin, Minglong and Song, Fang and
                  Yao, Penghui and Zhao, Mingnan},
  title        = {Parallel {Kac}'s Walk Generates {PRU}},
  howpublished = {arXiv:2504.14957},
  year         = {2025},
  note         = {Preprint}
}

@book{Talagrand2021,
  author    = {Talagrand, Michel},
  title     = {Upper and Lower Bounds for Stochastic Processes:
               Decomposition Theorems},
  edition   = {2},
  series    = {Ergebnisse der Mathematik und ihrer Grenzgebiete.
               3. Folge / A Series of Modern Surveys in Mathematics},
  volume    = {60},
  publisher = {Springer},
  address   = {Cham},
  year      = {2021},
  doi       = {10.1007/978-3-030-82595-9}
}

@article{PillaiSmith2017,
  author  = {Pillai, Natesh S. and Smith, Aaron},
  title   = {Kac's Walk on {$n$}-Sphere Mixes in {$n\log n$} Steps},
  journal = {The Annals of Applied Probability},
  volume  = {27},
  number  = {1},
  pages   = {631--650},
  year    = {2017},
  doi     = {10.1214/16-AAP1214}
}

\end{document}